\documentclass[11pt]{amsart}
\usepackage{fullpage,url,amssymb,mathrsfs,color}

\usepackage{palatino}

\numberwithin{equation}{section}
\usepackage{color}

\newcommand{\blank}[1]{{}}

\def\bbar#1{\setbox0=\hbox{$#1$}\dimen0=.2\ht0 \kern\dimen0 }
\newcommand{\defi}[1]{\textsf{#1}} 

\newenvironment{romanenum}{\hfill \begin{enumerate} }{\end{enumerate}}
\newenvironment{alphenum}{\hfill \begin{enumerate} }{\end{enumerate}}

  \newcommand{\FF}{{\mathbb F}}

 \newcommand{\QQ}{{\mathbb Q}}
\newcommand{\RR}{{\mathbb R}}
\newcommand{\ZZ}{{\mathbb Z}}

\newcommand{\kbar}{{\bbar{k}}}

\def\bbar#1{\setbox0=\hbox{$#1$}\dimen0=.2\ht0 \kern\dimen0 \overline{\kern-\dimen0 #1}}
\newcommand{\Qbar}{{\overline{\mathbb Q}}} 
\newcommand{\Kbar}{\bbar{K}}

  \renewcommand{\P}{{\mathfrak P}}
 \newcommand{\m}{{\mathfrak m}}

\newcommand{\calB}{{\mathcal B}}
\newcommand{\calC}{{\mathcal C}}
\newcommand{\calD}{{\mathcal D}}
 
\newcommand{\calF}{{\mathcal F}}
\newcommand{\calG}{{\mathcal G}}
\newcommand{\calH}{{\mathcal H}}

\newcommand{\calP}{{\mathcal P}}

\newcommand{\calU}{{\mathcal U}}

\newcommand{\OO}{{\mathcal O}}

\def\Bb{\mathcal B}

\DeclareMathOperator{\tr}{tr}

\DeclareMathOperator{\Tr}{Tr}

\DeclareMathOperator{\Frob}{Frob}

\DeclareMathOperator{\End}{End}

\DeclareMathOperator{\Aut}{Aut}
\DeclareMathOperator{\Gal}{Gal}
\DeclareMathOperator{\Ind}{Ind}

\DeclareMathOperator{\Cl}{Cl}

\DeclareMathOperator{\ord}{ord}

\newcommand{\GL}{\operatorname{GL}}
\newcommand{\SL}{\operatorname{SL}}

\newcommand{\Li}{\operatorname{Li}}

\newcommand{\Cent}{\operatorname{Cent}}

\def\CC{\mathbb C}
\def\p{\mathfrak{p}}

\newtheorem{theorem}{Theorem}[section]
\newtheorem{lemma}[theorem]{Lemma}
\newtheorem{corollary}[theorem]{Corollary}

\theoremstyle{definition}

\newtheorem{conjecture}[theorem]{Conjecture}

\theoremstyle{remark}
\newtheorem{remark}[theorem]{Remark}

\definecolor{webcolor}{rgb}{0,0,1}
\definecolor{webbrown}{rgb}{.6,0,0}
\usepackage[
        colorlinks,
        linkcolor=webbrown,  filecolor=webbrown,  citecolor=webbrown, 
        backref,
        pdfauthor={David Zywina}, 
        pdftitle={Bounds for the Lang-Trotter conjectures},
]{hyperref}
\usepackage[alphabetic,backrefs,lite]{amsrefs} 

\begin{document}

\title[Bounds for the Lang-Trotter conjectures]{Bounds for the Lang-Trotter conjectures}
\subjclass[2010]{Primary 11G05; Secondary 11N05, 11R45}
\author{David Zywina}
\address{Department of Mathematics, Cornell University, Ithaca, NY 14853, USA}
\email{zywina@math.cornell.edu}
\urladdr{http://www.math.cornell.edu/~zywina}

\begin{abstract}
For a non-CM elliptic curve $E/\QQ$, Lang and Trotter made very deep conjectures concerning the number of primes $p\leq x$ for which $a_p(E)$ is a fixed integer (and for which the Frobenius field at $p$ is a fixed imaginary quadratic field).    Under GRH, we use a smoothed version of the Chebotarev density theorem to improve the best known Lang-Trotter upper bounds of Murty, Murty and Saradha, and Cojocaru and David.
\end{abstract}

\maketitle

\section{Introduction}

\subsection{The Lang-Trotter conjectures}

Fix a non-CM elliptic curve $E$ defined over $\QQ$ and let $N_E$ be its conductor.   Take any prime $p\nmid N_E$.   Let $E_p$ be the reduction of $E$ modulo $p$; it is an elliptic curve over $\FF_p$.   Let $\pi_p$ be the Frobenius endomorphism of $E_p$.    We have $\pi_p^2-a_p(E) \pi_p + p=0$ for a unique integer $a_p(E)$.   We can also define $a_p(E)$ by the formula $a_p(E) =  |E_p(\FF_p)| - (p+1)$.  From Hasse, we know that $|a_p(E)|< 2\sqrt{p}$ and hence $\QQ(\pi_p)$ in $\End(E_p)\otimes_\ZZ \QQ$ is an imaginary quadratic field. 

Fix an integer $a$ and an imaginary quadratic field $k$.  We define the following functions of $x\geq 2$:
\begin{align*}
P_{E,a}(x) &:= \#\{ p\leq x : p\nmid N_E,\, a_p(E)=a \},\\
P_{E,k}(x) &:= \#\{ p\leq x : p\nmid N_E,\, \QQ(\pi_p) \cong k \}.
\end{align*}
Lang and Trotter made the following two conjectures concerning the asymptotics of $P_{E,a}(x)$ and $P_{E,k}(x)$, cf.~\cite{MR0568299}.

\begin{conjecture}[Lang-Trotter] \label{C:LT conj}
\begin{alphenum}
\item
There is an explicit constant $C_{E,a} \geq 0$ such that 
\[
P_{E,a}(x) \sim C_{E,a}\cdot \frac{x^{1/2}}{\log x}
\]
as $x\to \infty$.   When $C_{E,a}=0$, we interpret this asymptotic as meaning that $P_{E,a}(x)$ is a bounded function of $x$.
\item
There is an explicit constant $C_{E,k} > 0$ such that 
\[
P_{E,k}(x) \sim C_{E,k} \cdot \frac{x^{1/2}}{\log x}
\]
as $x\to \infty$.   
\end{alphenum}
\end{conjecture}

\subsection{Upper bounds}

In this paper, we are interested in improving the best known upper bounds on $P_{E,a}(x)$ and $P_{E,k}(x)$ as functions of $x$; we will summarize previous results in \S\ref{SS:previous results}.   Some bounds are conditional on the Generalized Riemann Hypothesis (GRH) for number fields.

\begin{theorem} \label{T:main a}
Let $E$ be a non-CM elliptic curve defined over $\QQ$ and let $a$ be an integer.   Assuming GRH, we have
\[
P_{E,a}(x) \ll_E \frac{x^{4/5}}{(\log x)^{3/5}} \quad \text{and}\quad P_{E,0}(x) \ll_E \frac{x^{3/4}}{(\log x)^{1/2}}.
\]
\end{theorem}

The best known unconditional bounds for $P_{E,a}(x)$ can be found in \S\ref{SS:previous results}.  We now give bounds for $P_{E,k}(x)$.

\begin{theorem} \label{T:main k}
Let $E$ be a non-CM elliptic curve  over $\QQ$ and let $k$ be an imaginary quadratic field.  
\begin{romanenum}
\item \label{T:main k i}
Assume GRH.  Then
\[
P_{E,k}(x) \ll_E  \,\, \frac{1}{h_k^{3/5}}  \frac{x^{4/5}}{ (\log x)^{3/5}} + x^{1/2}(\log x)^3,
\]
where $h_k$ is the class number of $k$.  In particular, $P_{E,k}(x) \ll_E  {x^{4/5}}/{(\log x)^{3/5}}$.
\item \label{T:main k ii}
There is a constant $c>0$, depending only on $E$ and $k$, such that 
\[
P_{E,k}(x) \ll_E \frac{x(\log\log x)^2}{(\log x)^2}
\]
whenever $x\geq c$.  In particular, $P_{E,k}(x) \ll_{E,k} {x(\log\log x)^2}/{(\log x)^2}$.
\end{romanenum}
\end{theorem}

Let $D_E(x)$ be the number of imaginary quadratic extensions $k$ of $\QQ$, in some fixed algebraic closure of $\QQ$, for which there exists a prime $p\leq x$ with $\QQ(\pi_p) \cong k$.   The following, which will be proved in \S\ref{S:D}, is an easy consequence of Theorem~\ref{T:main k}(\ref{T:main k i}).

\begin{corollary} \label{C:D theorem}
Let $E$ be a non-CM elliptic curve defined over $\QQ$.  Assuming GRH, we have
\[ 
D_E(x) \gg_E   \frac{x^{2/7}}{(\log x)^{10/7}} .
\]
\end{corollary}

This improves on the bound $D_E(x) \gg_E x^{1/14}/(\log x)^2$ from \cite{MR2464027}.   The explicit dependence of $k$ in Theorem~\ref{T:main k}(\ref{T:main k i}) is very important here.

\subsection{Some earlier results}  \label{SS:previous results}
We first describe bounds for $P_{E,a}(x)$.  Under GRH, Serre proved that $P_{E,a}(x) \ll_E x^{7/8}(\log x)^{1/2}$ and $P_{E,0}(x) \ll_E x^{3/4}$, cf.~\cite{MR644559}.    Under GRH, Murty, Murty and Saradha obtained the improved bound 
\[
P_{E,a}(x) \ll_E \frac{x^{4/5}}{(\log x)^{1/5}},
\] 
cf.~\cite{MR935007}.    In \cite{MR644559}, Serre proved (unconditionally) that $P_{E,a}(x)\ll_{E,\varepsilon}\,  x/(\log x)^{5/4-\varepsilon}$ for any $\varepsilon>0$.   The exponent $5/4$ was improved to $2$ by D.~Wan \cite{MR1062334}.  The best general unconditional bound for $P_{E,a}(x)$ is the bound 
\begin{align} \label{E:unconditional VKMurty}
P_{E,a}(x)\ll_E x \frac{(\log\log x)^2}{(\log x)^2}
\end{align}
of V.~K.~Murty \cite{MR1694997}.   For $a=0$, there is also the superior bound  $P_{E,0}(x) \ll_E x^{3/4}$ of Elkies, Kaneko and Murty, cf.~\cite{MR1144318}.
\\

We now describe bounds for $P_{E,k}(x)$.   In \cite{MR644559}*{p.~191}, Serre claimed without proof that $P_{E,k}(x) \ll_{E,k}\, x^\theta$ (under GRH) and $P_{E,k}(x) \ll_{E,k}\, x/(\log x)^{\gamma+1}$ (unconditionally) for some positive constants $\theta$ and $\gamma$.    Under GRH, Cojocaru, Fouvry and Murty \cite{MR2178556} showed that one could take $\theta=7/8$ and take any $\gamma>1/24$.  Under GRH, Cojocaru and David \cite{MR2464027} obtained the bound
\[
P_{E,k}(x) \ll_{E,k} \frac{x^{4/5}}{(\log x)^{1/5}}.\\
\]

Upper bounds for $P_{E,a}(x)$ and $P_{E,k}(x)$ are in general hard to improve.  The function of $x$ obtained indicates the strength of the methods used and often different methods will give the exact same bound.    For example, assume $E$ is semistable and  that the $L$-function for $E$ and its symmetric powers has analytic continuation and satisfies the appropriate analogue of the Riemann hypothesis, then Rouse and Thorner proved that $P_{E,0}(x) \ll_E  x^{3/4}/(\log x)^{1/2}$, cf.~\cite{Rouse-Thorner}.   This is the same bound as Theorem~\ref{T:main a} under GRH!

The goal of this paper is to push the upper bounds obtained using Chebotarev to the limit.   It is not clear to the author how to improve them without completely new ideas.

\subsection{Overview}

In \S\ref{S:Chebotarev}, we recall several effective versions of the Chebotarev density theorem.   Under GRH and Artin's holomorphy conjecture, we also give some improved Chebotarev upper bounds; the key point being that we can obtain superior error terms if we count primes using a smoothed weighting.

One can also prove a smoothed and \emph{unconditional} version of Chebotarev.   We have not done so because they did seem to lead to stronger Lang-Trotter bounds (in particular, we could not improve on the bound (\ref{E:unconditional VKMurty}) of V.K.~Murty).

In \S\ref{S:Galois representations}, we review some of the Galois representations associated to $E$ and $k$.   These representations play a role in the heuristics of Lang and Trotter in \cite{MR0568299}.    To understand their images we will need Serre's open image theorem and some class field theory.

We prove Theorem~\ref{T:main a} in \S\ref{S:main a}.   We follow the proof of Murty, Murty and Saradha in \cite{MR935007} and use our stronger Chebotarev bound.   We prove Theorem~\ref{T:main k} in \S\ref{S:main k}.   We again follow the general strategy of \cite{MR935007} though the groups are more complicated.

\begin{remark}
Let us remark how Theorem~\ref{T:main a} can be naturally generalized.   Suppose that $f(z) = \sum_{n\geq 1} a_n(f) e^{2\pi i n z}$ is a non-CM newform of integral weight $k \geq 2$ on $\Gamma_0(N)$ whose Fourier coefficients are rational numbers.  For a fixed integer $a$, we define $P_{f,a}(x)$ to be the number of primes $p\leq x$ for which $a_p(f)=a$.   

Using the methods in this paper, one can prove the Lang-Trotter bounds $P_{f,a}(x) \ll_{f} x^{4/5} (\log x)^{-3/5}$ and $P_{f,0}(x) \ll_f  x^{3/4} (\log x)^{-1/2}$ assuming GRH.   This is a generalization of Theorem~\ref{T:main a} since by modularity every elliptic curve $E/\QQ$ gives rise to a newform $f$ of weight $2$ on $\Gamma_0(N_E)$ satisfying $a_p(f) = a_p(E)$ for all primes $p\nmid N_E$.   Note that many of the bounds mentioned in \S\ref{SS:previous results} are proved for such general $f$.    (For $k\geq 3$, one should not expect the precise analog of Conjecture~\ref{C:LT conj} to hold.)

Our proof depends only on Galois representations and effective Chebotarev density theorems; we now make a few comments that describe what one needs to know about the Galois representations arising from $f$.  For each rational prime $\ell$, we know from Deligne that there is a Galois representation $\rho_{f,\ell}\colon \Gal_\QQ:=\Gal(\Qbar/\QQ) \to \GL_2(\FF_\ell)$ such that $\tr(\rho_{f,\ell}(\Frob_p)) \equiv a_p(f) \pmod{\ell}$ for all primes $p\nmid N\ell$.   For all but finitely many $\ell$, we will  have $\rho_{f,\ell}(\Gal_\QQ)\supseteq \SL_2(\FF_\ell)$, cf.~\cite{MR819838}.   
\end{remark}

\subsection*{Notation}

For two functions $f(x)$ and $g(x)$ of a real variable $x\geq 2$, we say that $f \ll g$ (or $g \gg f$) if there is a positive constant $C$ such that $|f(x)| \leq C|g(x)|$ for all $x\geq 2$.  We shall use $O(f)$ to denote an unspecified function $g$ with $g\ll f$.  We will always indicate the dependence of the implied
constant $C$ with subscripts on $\ll$ or $O$ (in particular, no subscripts indicates that the constant is absolute).  The logarithm integral is $\Li(x) := \int^x_2 (\log t)^{-1} \, dt$; it satisfies $\Li(x) \ll x/\log x$.

For a number field $K$, let $\OO_K$ be its ring of integers.   Let $\Sigma_K$ be the set of non-zero prime ideals of $\OO_K$.   For each $\p \in \Sigma_K$, let $N(\p)$ be the cardinality of the residue field $\OO_K/\p$.

For a number field $K$, let $\Kbar$ be a fixed algebraic closure of $K$.  Define the absolute Galois group $\Gal_K:=\Gal(\Kbar/K)$.   For each $\p \in \Sigma_K$, let $\Frob_\p \in \Gal_K$ be a Frobenius automorphism for the prime $\p$.  If $\rho\colon \Gal_K \to G$ is a representation unramified at $\p$, then $\rho(\Frob_\p)$ is a well-defined conjugacy class.

\subsection*{Acknowledgements}
This paper uses parts of an unpublished preprint that benefited from helpful comments from Bjorn Poonen.  Thanks also to Theodore Hui and the referee for making several corrections.

\section{Chebotarev bounds} \label{S:Chebotarev}

Fix a Galois extension of number fields $L/K$ with Galois group $G$.    Define
\[
M(L/K):= 2 [L:K] \cdot d_K^{1/[K:\QQ]} \cdot {\prod}_{p\in \calP(L/K)}\,p,
\]
where $d_K$ is the absolute discriminant of $K$ and $\calP(L/K)$ is the set of rational primes $p$ that are divisible by some $\p \in \Sigma_K$ that ramifies in $L$.

We say that $L/K$ satisfies \defi{Artin's Holomorphy Conjecture} (AHC) if for each irreducible character $\chi \colon G \to \CC$, the Artin $L$-function $L(s,\chi)$ extends to a function analytic on the whole complex plane except at $s=1$ when $\chi=1$.   If $G=\Gal(L/K)$ is abelian, then AHC is known to hold for $L/K$; the Artin $L$-function then agrees with a Hecke $L$-function that has the required properties.  The \defi{Generalized Riemann Hypothesis} (GRH) for the field $L$ asserts that any zero $\rho$ of the Dedekind $L$-function of the field $L$ with $0\leq \operatorname{Re}(s) \leq 1$ satisfies $\operatorname{Re}(\rho)=1/2$.   We say that GRH holds if it holds for all number fields.

\subsection{Chebotarev density theorem} \label{SS:basic Cheb}

Let $\varphi\colon G \to \CC$ be a class function.  For each prime $\p\in \Sigma_K$, choose any $\P\in \Sigma_L$ dividing $\p$.   We then have a distinguished (arithmetic) Frobenius element $\sigma_{\P} \in D_\P/I_\P$, where $D_\P$ and $I_\P$ are the decomposition and inertia subgroups of $G$, respectively, at $\P$.    For each integer $m\geq 1$, we define
\[
\varphi(\Frob_\p^m) := \frac{1}{|I_\P|} \sum_{\substack{g\in D_\P,\\ gI_\P = \sigma_\P^m \in D_\P/I_\P}} \varphi(g).
\]
As the notation suggests, $\varphi(\Frob_\p^m)$ is independent of the choice of $\P$.  For $\p$ unramified in $L$, this definition agrees with the value of $\varphi$ on the conjugacy class $\Frob_\p^m$ of $G$.  For $x\geq 2$, define
\[
\pi_\varphi(x) := \sum_{\substack{\p\in\Sigma_K \text{ unramified in $L$}\\N(\p)\leq x}} \varphi(\Frob_\p) \text{ \quad and \quad }
\widetilde{\pi}_\varphi(x) := \sum_{\substack{\p\in\Sigma_K, m\geq 1\\ N(\p^m)\leq x}} \frac{1}{m} \varphi(\Frob_\p^m).
\]

Now let $C$ be a subset of $G$ that is stable under conjugation and let $\delta_C \colon G \to \{0,1\}$ be the class function such that $\delta_C(g)=1$ if and only if $g\in C$.    We define 
\[
\pi_C(x,L/K) := \pi_{\delta_C}(x)\quad \text{ and }\quad \widetilde\pi_C(x,L/K) := \widetilde{\pi}_{\delta_C}(x).
\]
It is often more convenient to use $\widetilde\pi_C(x,L/K)$ since it has better functorial properties, cf.~\S\ref{SS:functorial}.

The \defi{Chebotarev density theorem} says that
\begin{equation} \label{E:Chebotarev first}
\pi_C(x,L/K) \sim \frac{|C|}{|G|} \Li(x) 
\end{equation}
as $x\to +\infty$.   An effective form of Chebotarev is a version with an explicit error term.   The following theorems of Murty, Murty and Saradha give effective versions.

\begin{theorem}  \label{T:effective}
\begin{romanenum}
\item \label{T:effective i}
Suppose that AHC holds for $L/K$ and that GRH holds for $L$.  Then
\[
\pi_C(x,L/K) = \frac{|C|}{|G|}\Li(x) + O\Big( |C|^{1/2}\, [K:\QQ] \, x^{1/2}  \log  (M(L/K)x) \Big).
\]
\item \label{T:effective ii}
Assume that the group $G$ is abelian.   There are absolute constants $b,c>0$ such that if $\log x \geq b [K:\QQ] \log^2 M(L/K)$, then
\begin{align*}
&\Big| \pi_C(x,L/K) - \frac{|C|}{|G|} \Li(x) \Big| \\
\leq & \frac{|C|}{|G|} \Li(x^{\beta_L}) + O\bigg( |C|^{1/2} [K:\QQ] x \exp\bigg(-\frac{c(\log x)^{1/2}}{[K:\QQ]^{1/2}} \bigg) \cdot\log^2(M(L/K)x) \bigg),
\end{align*}
where $\beta_L$ is the possible exceptional zero of the Dedekind $L$-function of the field $L$ (it would be real and satisfy $1/2< \beta_L <1$).   The term $\frac{|C|}{|G|} \Li(x^{\beta_L})$ is present only when $\beta_L$ exists. 	
\end{romanenum}
\end{theorem}
\begin{proof}
Part (\ref{T:effective i}) is Corollary~3.7 of \cite{MR935007}.     Part (\ref{T:effective ii}) is a special case of Theorem~4.6 of \cite{MR1694997} where the group is abelian.  In the notation of \cite{MR1694997},  we have $H=1$ and $G/H$  abelian, so $d_{G/H}=1$ and $|\chi_{G/H}(\bbar{C})|\leq |C|$.
\end{proof}

\subsection{Smoothed Chebotarev}

We will prove the following smoothed analogue of Chebotarev in \S\ref{S:smoothed proof}.

\begin{theorem} \label{T:smoothed}
Let $L/K$ be a finite Galois extension of number fields with Galois group $G$.  Assume that AHC holds for the extension $L/K$ and that GRH holds for $L$.    Let $C$ be a subset of $G$ stable under conjugation.    

Take any smooth function $f\colon (0,\infty) \to \RR$ with compact support.  Then for $x\geq 2$, we have
\begin{align*}
&\sum_{\p} \delta_C(\Frob_\p) \log N(\p) \cdot f(N(\p)/x) 
=  \frac{|C|}{|G|} \int^\infty_0 f(t) dt\, \cdot x + O_f\Big(|C|^{1/2}  [K:\QQ]  x^{1/2} \log M(L/K) \Big),
\end{align*}
where the sum is over all the primes $\p \in \Sigma_K$ that are unramified in $L$.
\end{theorem}

In our application, we are only interested in asymptotic upper bounds for $\pi_C(x,L/K)$, so there is no harm in counting using a smoothed weight.  Under GRH, the following gives an upper bound that cannot be deduced from Theorem~\ref{T:effective}(\ref{T:effective i}).

\begin{theorem} \label{T:Chebotarev upper bound}
Let $L/K$ be a Galois extension of number fields with Galois group $G$.  Assume that AHC holds for $L/K$ and that GRH holds for $L$.  Let $C$ be a subset of $G$ that is stable under conjugation. Then
\[
\pi_C(x,L/K) \ll \frac{|C|}{|G|}\,  \frac{x}{\log x} + |C|^{1/2}\, [K:\QQ] \frac{x^{1/2}}{\log x}  \log  M(L/K).
\]
\end{theorem}
\begin{proof}
First fix a smooth function $f\colon (0,\infty)\to \RR$ with compact support that is non-negative and satisfies $f(t) \geq 1$ for all $1/2\leq t \leq 1$.   For every $x\geq 2$, define
\[
A(x):= \sum_{\substack{\p \in \Sigma_K,\, \sqrt{x}\leq N(\p) \leq x \\ \p \text{ unramified in $L$}}} \delta_C(\Frob_\p) \log N(\p) \quad \text{and}\quad \Pi(x) := \sum_{\p \in \Sigma_K, \, x/2\leq N(\p) \leq x} \delta_C(\Frob_\p) \log N(\p).
\]
By our choice of $f$, we have $\Pi(x) \leq \sum_{\p \in \Sigma_K} \delta_C(\Frob_\p) \log N(\p) \cdot f(N(\p)/x)$.  By Theorem~\ref{T:smoothed}, we have $\Pi(x)\ll_f \tfrac{|C|}{|G|} x + |C|^{1/2}  [K:\QQ]  x^{1/2} \log M(L/K)$.

Let $m\geq 1$ be the largest integer for which $x/2^m \geq \sqrt{x}$.  We have $A(x)\leq \sum_{i=0}^m \Pi(x/2^i)$, so our bound for $\Pi(x)$ gives
\begin{align*}
A(x) & \ll_f  \frac{|C|}{|G|} x \sum_{i=0}^m \frac{1}{2^i} + |C|^{1/2}  [K:\QQ]  x^{1/2} \log M(L/K) \sum_{i=0}^m \frac{1}{2^{i/2}} \\
& \ll  \frac{|C|}{|G|} x + |C|^{1/2}  [K:\QQ]  x^{1/2} \log M(L/K).
\end{align*}
There are at most $[K:\QQ]$ primes $\p$ dividing any rational prime $p$, so 
\[
|\{\p \in \Sigma_K : N(\p)\leq \sqrt{x}\}| \leq [K:\QQ] \cdot |\{p: p\leq \sqrt{x}\}| \ll [K:\QQ] \tfrac{\sqrt{x}}{\log \sqrt{x}}.
\]
 Therefore,
\begin{align*}
\pi_C(x,L/K) &\ll  A(x)/\log(\sqrt{x}) + [K:\QQ] x^{1/2}/\log(\sqrt{x})\\
& \ll_f   \frac{|C|}{|G|} \frac{x}{\log x} + |C|^{1/2}  [K:\QQ]  \frac{x^{1/2}}{\log x} \log M(L/K) + [K:\QQ] \frac{x^{1/2}}{\log x}.
\end{align*}
The theorem now follows if $|C|\neq 0$.  If $|C|=0$, then the theorem is trivial.
\end{proof}

For future use, we also give the following consequence of Theorem~\ref{T:smoothed}.

\begin{corollary} \label{C:cheb existence}
Fix notation and assumptions as in Theorem~\ref{T:smoothed}.  Assume that $C\neq \emptyset$.  There is an absolute constant $c>0$ such that if 
\[
x \geq c \frac{\;|G|^2}{|C|} [K:\QQ]^2 \log^2 M(L/K), 
\]
then there is a prime $\p \in \Sigma_K$ unramified in $L$ with $x/2 \leq N(\p) \leq x$ such that $\delta_C(\Frob_\p)=1$.  
\end{corollary}
\begin{proof}
Let $f\colon (0,\infty)\to \RR$ be a non-negative smooth function with compact support whose support is in the interval $[1/2,1]$ and is non-zero.   Suppose that there are no primes $\p \in \Sigma_K$ with $x/2\leq N(\p) \leq x$ such that $\p$ is unramified in $L$ and $\delta_C(\Frob_\p)=1$.   The sum in Theorem~\ref{T:smoothed} is then zero, so
\[
 \frac{|C|}{|G|} x \ll_f |C|^{1/2}  [K:\QQ]  x^{1/2} \log M(L/K);
\]
note that the integral $\int^\infty_0 f(t) dt$ is positive by our choice of $f$.  Rearranging, we deduce that $x \ll_f |G|^2/|C|\cdot [K:\QQ]^2 \log^2 M(L/K)$.   We obtain the desired contradiction by choosing the constant $c$ in the statement of the corollary sufficiently large.
 \end{proof}

\subsection{Functorial properties} \label{SS:functorial}

Let $H$ be a subgroup of $G$.  Take any class function $\varphi\colon H \to \CC$.  As above, we can define $\widetilde{\pi}_\varphi(x)$; note that $H$ is the Galois group of the extension $L/L^H$.  Define the induced function
\[
\Ind_H^G\varphi \colon G \to \CC,\quad g\mapsto \frac{1}{|H|} \sum_{{t\in G},\,{t^{-1}gt\in H}} \varphi(t^{-1}gt);
\]
it is a class function of $G$.

\begin{lemma} \label{L:invariance under induction}

\begin{romanenum}
\item \label{L:invariance under induction i}
Let $H$ be a subgroup of $G$ and let $\varphi$ be a class function of $H$.  Then $\widetilde{\pi}_{\Ind^G_H \varphi}(x) = \widetilde{\pi}_\varphi(x)$.
\item \label{L:invariance under induction ii}
Let $N$ be a normal subgroup of $G$ and let $\varphi'$ be a class function of $G/N$.   Then $\widetilde{\pi}_{\varphi'}(x)=\widetilde{\pi}_{\varphi}(x)$,
where $\varphi$ is the function obtained by composing the projection $G\to G/N$ with $\varphi'$.
\end{romanenum}
\end{lemma}
\begin{proof}
See Proposition~8 of \cite{MR644559}.
\end{proof}

\begin{lemma}
\label{L:invariance C}
\begin{romanenum}
\item \label{L:invariance C i}
Let $H$ be a subgroup of $G$ and let $C$ be a subset of $G$ stable under conjugation.   Suppose that every element of $C$ is conjugate to some element of $H$. Then
\[
\widetilde{\pi}_C(x,L/K) \leq \widetilde{\pi}_{C\cap H}(x,L/L^H).
\]
\item \label{L:invariance C ii}
Let $N$ be a normal subgroup of $G$ and let $C$ be a subset of $G$ stable under conjugation that satisfies $NC \subseteq C$.   Then
\[
\widetilde{\pi}_C(x,L/K) = \widetilde{\pi}_{C'}(x,L^N/K),
\]
where $C'$ is the image of $C$ in $G/N=\Gal(L^N/K)$.  
\end{romanenum}
\end{lemma}
\begin{proof}
Part (\ref{L:invariance C ii}) is an immediate consequence of Lemma~\ref{L:invariance under induction}(\ref{L:invariance under induction ii}); note that $\delta_{C}$ is equal to the projection $G\to G/N$ composed with $\delta_{C'}$.

We now prove (\ref{L:invariance C i}).  By assumption on $C$, there is a set $S\subseteq H$ for which we have a disjoint union $C=\cup_{s\in S} C_G(s)$, where $C_G(s)$ is the conjugacy class of $s$ in $G$.   For each $s\in S$, let $C_H(s)$ be the conjugacy class of $s$ in $H$.  We have $(\Ind_H^G \delta_{C_H(s)})(g)=0$ for $g\in G -C_G(s)$, so 
\[
\Ind_H^G \delta_{C_H(s)}=\lambda_s \cdot \delta_{C_G(s)}
\] 
for some $\lambda_s$.  Therefore, $\widetilde\pi_{C_H(s)}(x,L/L^H) = \lambda_s \widetilde\pi_{C_G(s)}(x,L/K)$ by Lemma~\ref{L:invariance under induction}(\ref{L:invariance under induction i}).   We have 
\[
\widetilde\pi_{C}(x,L/K)=\sum_{s\in S} \widetilde\pi_{C_G(s)}(x,L/K) =\sum_{s\in S} \lambda_s^{-1} \widetilde\pi_{C_H(s)}(x,L/L^H).
\] 
Using Frobenius reciprocity, cf.~\cite{MR0450380}*{Theorem~13}, we have
\begin{align*} 
\lambda_s  \cdot {|C_G(s)|}/{|G|} = \big<\lambda_s \cdot  \delta_{C_G(s)}, 1_G  \big>_G 
= \big< \Ind_H^G \delta_{C_H(s)}, 1_G  \big>_G = \big< \delta_{C_H(s)}, 1_H \big>_H = {|C_H(s)|}/{|H|}.
\end{align*}
We have $|C_G(s)|=|G|/|\Cent_G(s)|$ and $|C_H(s)|=|H|/|\Cent_H(s)|$, where $\Cent_H(s)$ and $\Cent_G(s)$ are the centralizers of $s$ in $G$ and $H$, respectively.  Therefore, $\lambda_s^{-1} = [\Cent_G(s): \Cent_H(s)]^{-1} \leq 1$ and hence
\[
\widetilde\pi_{C}(x,L/K)\leq \sum_{s\in S} \widetilde\pi_{C_H(s)}(x,L/L^H) \leq \widetilde{\pi}_{C\cap H}(x,L/L^H).   \qedhere
\] 
\end{proof}

In our applications, we will use Lemma~\ref{L:invariance C} to reduce our computations of $\pi_C(x,L/K)$ to the case where $G$ is abelian.  The following says that $\widetilde\pi_C(x,L/K)$ is a good approximation of $\pi_C(x,L/K)$.  

\begin{lemma} \label{L:good approximation}
For any subset $C$ of $G$ stable under conjugation, we have 
\[
\widetilde\pi_C(x,L/K) = \pi_C(x,L/K) +O\big( [K:\QQ]\big( \tfrac{x^{1/2}}{\log x} + \log M(L/K)\big) \big).
\]
\end{lemma}
\begin{proof}
Let $\pi_K(x)$ be the number of $\p \in \Sigma_K$ for which $N(\p)\leq x$.    Since there are at most $[K:\QQ]$ primes $\p\in \Sigma_K$ dividing any $p$, we have $\pi_K(x) \ll [K:\QQ] \cdot x/\log x$.   For each $\p \in \Sigma_K$, let $\deg \p$ be the integer for which $N(\p)=p^{\deg \p}$, where $p$ is the prime divisible by $\p$.   Define the sums
\[
B_1:= \sum_{m\geq 2} \sum_{\substack{\p\in \Sigma_K \\ N(\p)^m \leq x}} \frac{1}{m}, \quad 
B_2:= \sum_{\substack{\p\in \Sigma_K, \, \deg \p > 1\\ N(\p)\leq x}} 1 \quad \text{ and } \quad
B_3:= \sum_{\substack{\p \in \Sigma_K \text{ ramified in $L$} \\ \deg \p =1 ,\, N(\p)\leq x} } 1.
\]
We have $0\leq \widetilde\pi_C(x,L/K)-\pi_C(x,L/K) \leq B_1+B_2+B_3$, so it suffices to bound the $B_1$, $B_2$ and $B_3$.

Let $M\geq 1$ be the largest integer for which $x^{1/M} \geq 2$.  We have
\[
B_1 \leq \sum_{m=2}^M \frac{1}{m} \pi_K(x^{1/m}) \leq [K:\QQ] \sum_{m=2}^M \frac{1}{m} \frac{x^{1/m}}{\log(x^{1/m})} \leq [K:\QQ] \frac{x^{1/2}}{\log x} \sum_{m=2}^M \frac{1}{m^2} \ll [K:\QQ] \frac{x^{1/2}}{\log x}.
\]
We have $B_2 \leq \sum_{p \leq \sqrt{x}} [K:\QQ] \ll [K:\QQ] \cdot x^{1/2}/\log x$.  If $\p \in \Sigma_K$ is a prime with $\deg \p =1$ that ramifies in $L$, then $N(\p) \in \calP(L/K)$, where $\calP(L/K)$ is the set from the definition of $M(L/K)$.   Since there are at most $[K:\QQ]$ primes $\p\in \Sigma_K$ dividing any $p$, we have $B_3 \leq [K:\QQ] \sum_{p\in \calP(L/K)} \log p \leq  [K:\QQ] \log M(L/K)$.
\end{proof}

\section{Galois representations}   \label{S:Galois representations}

\subsection{Elliptic curves}

Fix a non-CM elliptic curve $E$ defined over $\QQ$.  For each prime $\ell$, let $E[\ell]$ be the $\ell$-torsion subgroup of $E(\Qbar)$; it is a free $\FF_\ell$-module of rank $2$.    There is a natural action of $\Gal_\QQ$ on $E[\ell]$ that respects the group structure and can be expressed in terms of a Galois representation
\[
\rho_{E,\ell}\colon \Gal_\QQ \to \Aut_{\FF_\ell}(E[\ell]) \cong \GL_2(\FF_\ell).
\]
The representation $\rho_{E,\ell}$ is unramified at all primes $p\nmid N_E \ell$ and we have
\[
\det(xI - \rho_{E,\ell}(\Frob_p))\equiv x^2- a_p(E) x + p \pmod{\ell}.
\]
We have $\det \circ \rho_{E,\ell} = \chi_\ell$, where $\chi_\ell\colon \Gal_\QQ \to \FF_\ell^\times$ is the representation for which $\sigma(\zeta)=\zeta^{\chi_\ell(\sigma)}$ for all $\ell$-th roots of unity $\zeta\in \Qbar$.  The following is an important theorem of Serre, cf.~\cite{MR0387283}.

\begin{theorem}[Serre]   \label{T:Serre}
The representation $\rho_{E,\ell}$ is surjective for all but finitely many primes $\ell$.  
\end{theorem}

\subsection{Some class field theory}  Fix an imaginary quadratic field $k$ and let $\calH$ be its Hilbert class field.   Denote the ring of integers of $k$ by $\OO$.    Take any non-zero prime ideal $\p$ of $\OO$.   Let $v_\p\colon k^\times \twoheadrightarrow \ZZ$ be the surjective discrete valuation corresponding to $\p$ which we extend by setting $v_\p(0)=+\infty$.

Fix an integer $m\geq 1$.  Let $I_k^{m}$ be the group of fractional ideals of $k$ generated by prime ideals $\p \nmid m$ of $\OO$.     Let 
\[
\iota\colon k_m:= \{ a\in k^\times: v_\p(a)=0 \text{ for all }\p|m\} \to I_k^{m}
\]
be the homomorphism that takes an element of $k_m$ to the fractional ideal of $k$ it generates.    The \defi{ray class group modulo $m$} is 
the group
\[
\Cl_m := I_k^{m}/\iota(k_{m,1}),
\]
where $k_{m,1} := \{ a\in k_m: v_\p(a-1) \geq v_\p(m) \text{ for all }\p|m \}$.    Note that $\Cl_{1}$ is the usual class group of $\OO$ which we will also denote by $\Cl_k$.

By class field theory, there is a continuous homomorphism 
\[
\widetilde\psi_{k,m} \colon \Gal_k \to \Cl_m
\]
such that for each prime $\p\nmid m$ of $\OO$, $\widetilde\psi_{k,m}$ is unramified at $\p$ and $\widetilde\psi_{k,m}(\Frob_\p)$ is the class of $\Cl_m$ represented by $\p$.   The map $\widetilde\psi_{k,m}$ is clearly surjective. 

  The homomorphism $\widetilde\psi_{k,1}$ is unramified at all non-zero prime ideals $\p$ of $\OO$ and has image $\Cl_k$.   By class field theory, we deduce that the fixed field in $\kbar$ of $\ker(\widetilde\psi_{k,1})$ is $\calH$, i.e., the Hilbert class field of $k$.

\begin{lemma} \label{L:class group sequence}
If $m\geq 5$, then there is an exact sequence of groups
\[
1 \to \OO^\times \xrightarrow{\alpha_m}  (\OO/m\OO)^\times \xrightarrow{\beta_m} \Cl_{m} \xrightarrow{\gamma_m}  \Cl_k \to 1,
\]
where $\alpha_m$ is reduction modulo $m$, $\beta_m$ maps a coset $a+m\OO$ to the class of $\Cl_{\m}$ containing $a\OO$, and $\gamma_m$ is induced by the natural map $I_k^{m}\to \Cl_k$.
\end{lemma}
\begin{proof}
Let $\bbar{\iota}\colon k_m/k_{m,1} \to \Cl_m$ be the group homomorphism induced from $\iota\colon k_m \xrightarrow{\iota} I_k^m$.  Since $\ker(\iota)=\OO^\times$, we have an exact sequence
\[
\OO^\times \to k_m/k_{m,1} \xrightarrow{\bbar \iota} \Cl_m.
\]
We have $\OO^\times \cap k_{m,1} = 1$ since $m\geq 5$.   The group $I_k^m/\iota(k_m)$ is naturally isomorphic to $\Cl_k$.   We thus have an exact sequence
\begin{equation} \label{E:last easy ANT}
1\to \OO^\times \to k_m/k_{m,1} \xrightarrow{\bbar \iota} \Cl_m \to \Cl_k \to 1.
\end{equation}

Finally, we explain why $(\OO/m\OO)^\times$ is isomorphic to $k_m/k_{m,1}$.   We have a natural inclusion $k_m\hookrightarrow \OO_{m}^\times$, where $\OO_{m}$ is the $m$-adic completion of $\OO$.  Composing with the reduction modulo $m$ map gives a group homomorphism, $f\colon k_\m \to (\OO_{m}/m\OO_{m})^\times = (\OO/m\OO)^\times$.   The kernel of $f$ is $k_{m,1}$ and it is surjective by weak approximation.   We thus an induced isomorphism $\overline{f} \colon k_m/k_{m,1} \stackrel{\sim}{\to} (\OO/m\OO)^\times$.  Identifying $k_m/k_{m,1}$ in (\ref{E:last easy ANT}) by $(\OO/m\OO)^\times$ via the isomorphism $\overline{f}$, gives an exact sequence that agrees with the one in the statement of the lemma.
\end{proof}

We now focus on the case where $m$ is a prime $\ell\geq 5$.  By Lemma~\ref{L:class group sequence}, we may view $\OO^\times$ as a subgroup of $(\OO/\ell\OO)^\times$.      With $\gamma_\ell$ as in Lemma~\ref{L:class group sequence}, we have $\widetilde\psi_{k,\ell} = \gamma_\ell \circ \widetilde\psi_{k,1}$.   So by restricting $\widetilde\psi_{k,\ell}$ to $\Gal_\calH$ and using Lemma~\ref{L:class group sequence}, we obtain a surjective homomorphism
\[
\psi_{k,\ell} \colon \Gal_\calH \to  (\OO/\ell\OO)^\times/\OO^\times.
\]

\begin{lemma} \label{L:psim description}
Fix a prime $p\nmid \ell$ that splits completely in $\calH$.   Let $\P$ be a prime ideal of $\OO_\calH$ that divides $p$.    The representation $\psi_{k,\ell}$ is unramified at $\P$ and satisfies 
\[
\psi_{k,\ell}(\Frob_\P) = (\pi+ \ell\OO)\cdot \OO^\times  \in \left(\OO/\ell\OO\right)^\times/\OO^\times,
\]
where $\pi$ is a generator of the prime ideal $\P \cap \OO$ of $\OO$.
\end{lemma}
\begin{proof}
Set $\p:=\P \cap \OO$; it is a prime that splits completely in $\calH$.  Since $\p\nmid \ell$, $\widetilde\psi_{k,\ell}$ is unramified at $\p$ and  $\widetilde\psi_{k,\ell}(\Frob_\p)=[\p] \in \Cl_\ell$.    That $\p$ splits completely in $\calH$ implies that $\gamma_\ell([\p])=1$ and hence that $\p$ is indeed principal, say $\p=\pi \OO$.   We have $\gamma_\ell([\p])$, so we may identify $\widetilde\psi_{k,\ell}(\Frob_\p)=[\p]$ with an element of $(\OO/\ell\OO)^\times/\OO^\times$ as viewed above as a subgroup of $\Cl_\ell$; in particular, we identify $\widetilde\psi_{k,\ell}(\Frob_\p)$ with the coset represented by $\pi$.     Finally, since $\p$ splits completely in $\calH$ we find that $\widetilde\psi_{k,\ell}(\Frob_\P)=\widetilde\psi_{k,\ell}(\Frob_\p)$.   Therefore, $\psi_{k,\ell}(\Frob_\p)$ is represented by $\pi$ as desired.
\end{proof}

Let $N_{k/\QQ} \colon\left(\OO/\ell\OO\right)^\times/\OO^\times \to \FF_\ell^\times$ be the homomorphism induces by the usual norm map $N_{k/\QQ}\colon k\to \QQ$; it is well defined since the norm map takes value $1$ on $\OO^\times$.  

\begin{lemma} \label{L:cyclotomic-CFT}
The homomorphism $N_{k/\QQ}\circ \psi_{k,\ell} \colon \Gal_\calH \to \FF_\ell^\times$ agrees with $\chi_\ell|_{\Gal_\calH}$.
\end{lemma}
\begin{proof}
Take any prime $\P | p$ as in the statement of Lemma~\ref{L:psim description}.   It suffices to show that $N_{k/\QQ}(\psi_{k,\ell}(\Frob_\P))= \chi_\ell(\Frob_\P)$ since such primes $\P$ of $\OO_\calH$ have density $1$.    Since $p$ splits completely in $\calH$, we have $\chi_\ell(\Frob_\P)\equiv N(\p)=p \pmod{\ell}$.

By Lemma~\ref{L:psim description}, we have $N_{k/\QQ}(\psi_{k,\ell}(\Frob_\P))\equiv N_{k/\QQ}(\pi) \pmod{\ell}$, where $\pi\in \OO$ is a generator of the ideal $\P\cap\OO$.  Since $p$ splits completely in $\calH$, and hence also $k$, we have $N_{k/\QQ}(\pi)=p$.   Therefore, $N_{k/\QQ}(\psi_{K,\ell}(\Frob_\P)) \equiv p \equiv \chi_\ell(\Frob_\P) \pmod{\ell}$.
\end{proof}

\subsection{Mixed representations} \label{SS:mixed}

Let $E$ be a non-CM elliptic curve over $\QQ$.  Let $k$ be an imaginary quadratic field, let $\OO$ be its ring of integers, and let $\calH$ be its Hilbert class field.  

Fix a prime $\ell\geq 5$.  We have Galois representations $\rho_{E,\ell}$ and $\psi_{k,\ell}$ from the previous sections.  By Lemma~\ref{L:cyclotomic-CFT}, we have $\det\circ \rho_{E,\ell}|_{\Gal_\calH} = N_{k/\QQ}\circ \psi_{k,\ell}$.   We thus have a well-defined Galois representation
\[
\Psi_\ell\colon \Gal_{\calH} \to \calG,\quad \sigma\mapsto (\rho_{E,\ell}(\sigma),\psi_{k,\ell}(\sigma)),
\]
where $\calG:= \{ (A,u)\in\GL_2(\FF_\ell) \times \left((\OO/\ell\OO)^\times/\OO^\times \right) : \det(A)= N_{k/\QQ}(u) \}$.

The trace map $\Tr_{k/\QQ}\colon k\to \QQ$ induces a linear map $\Tr_{k/\QQ}\colon \OO/\ell\OO \to \FF_\ell$.  For $u\in (\OO/\ell\OO)^\times/\OO^\times$, $\Tr_{k/\QQ}(u)$ is a subset of $\FF_\ell$ of cardinality at most $|\OO^\times|$.

\begin{lemma} \label{L:LT2 key}
Let $p\nmid N_E$ be a prime for which $E$ has ordinary reduction at $p$ and for which $\QQ(\pi_p)$ is isomorphic to $k$.
\begin{romanenum}
\item \label{L:LT2 key i}
The prime $p$ splits completely in $\calH$.
\item \label{L:LT2 key ii}
Take any prime $\P \in \Sigma_\calH$ dividing $\pi_p \OO$.  Then for any prime $\ell\geq 5$ not equal to $p$, the representations $\rho_{E,\ell}$ and $\psi_{k,\ell}$ are unramified at $\P$ and 
\[
\tr(\rho_{E,\ell}(\Frob_\P))\in \Tr_{k/\QQ}(\psi_{k,\ell}(\Frob_\P)).
\]
\end{romanenum}
\end{lemma}
\begin{proof}
To ease notation, set $k=\QQ(\pi_p)$.   Since $\pi_p$ is a root of $x^2-a_p(E)x+p$, we have $a_p(E)=\Tr_{k/\QQ}(\pi_p)$ and $p=N_{k/\QQ}(\pi_p)$.  The equality $p=N_{k/\QQ}(\pi_p)$ implies that $p$ is either split or ramified in $k$, so  $p\OO = \p\cdot\p^\tau$, where $\p:=\pi_p \OO$ and $\tau$ is the non-trivial automorphism of $k$.  

We claim that $p$ splits in $k$.  Suppose otherwise that $p$ is ramified in $k$.  We then have
\[
a_p(E) =\Tr_{k/\QQ}( \pi_p )= \pi_p+ \pi_p^\tau  \in \p + \p^\tau =\p
\]
and hence $a_p(E)\equiv 0 \pmod{p}$ which contradicts our assumption that $E$ has ordinary reduction at $p$.

Since $\p$ and $\p^\tau$ are principal ideals in $\OO$, the prime $p$ splits completely in $\calH$.   This completes the proof of (\ref{L:LT2 key i}).

Now take any prime $\ell\nmid 6p$ and any $\P\in \Sigma_\calH$ that divides $\p = \pi_p \OO$.  Since $p$ splits completely in $\calH$, we have $\OO_\calH/\P = \FF_p$.   Therefore, $\rho_{E,\ell}$ is unramified at $\P$ and we have 
\[
\tr(\rho_{E,\ell}(\Frob_\P)) = \tr(\rho_{E,\ell}(\Frob_p)) \equiv a_p(E) \pmod{\ell}
\]
By Lemma~\ref{L:psim description}, we have $\psi_{k,\ell}(\Frob_\P) = (\pi_p + \ell\OO) \cdot \OO^\times$.  The image of $a_p(E)=\Tr_{k/\QQ}(\pi_p)$ in $\FF_\ell$ thus belongs to $\Tr_{k/\QQ}(\psi_{k,\ell}(\Frob_\P))$.    Therefore, $\tr(\rho_{E,\ell}(\Frob_\P))\in \Tr_{k/\QQ}(\psi_{k,\ell}(\Frob_\P))$ as claimed.
\end{proof}

The representation $\Psi_\ell$ is surjective for all sufficiently large $\ell$.

\begin{lemma} \label{L:intersection} 
If $\rho_{E,\ell}$ is surjective, then the representation $\Psi_\ell$ is surjective.
\end{lemma}
\begin{proof}
Let $p_1\colon \calG \to \GL_2(\FF_\ell) $ and $p_2\colon \calG \to (\OO/\ell\OO)^\times/\OO^\times$ be the projection maps.   Let $H$ be a subgroup of $\calG$ with $p_1(H)\supseteq\SL_2(\FF_\ell)$ and $p_2(H)=(\OO/\ell\OO)^\times/\OO^\times$.  

We claim that $H=\calG$.  For a finite group $G$, let $G'$ be the commutator subgroups of $G$.  Since $p_1(H)$ contains $\SL_2(\FF_\ell)$, we have  
\[
\SL_2(\FF_\ell)' \subseteq p_1(H)' \subseteq \GL_2(\FF_\ell)'=\SL_2(\FF_\ell).
\]
The group $\SL_2(\FF_\ell)$ is perfect since $\ell\geq 5$, so  $p_1(H')=p_1(H)' = \SL_2(\FF_\ell)$.  We have $p_1(H')=\SL_2(\FF_\ell)$ and $p_2(H')=((\OO/\ell\OO)^\times/\OO^\times)'=\{1\}$, so $H' = \SL_2(\FF_\ell) \times \{1\}$.    The group $H'$ is normal in $\calG$ and $p_2$ induces an isomorphism $\calG/H' \xrightarrow{\sim} (\OO/\ell\OO)^\times/\OO^\times$.   Since $p_2|_H$ is surjective, we deduce that the natural map $H/H'\hookrightarrow \calG/H'$ is surjective and hence $H=\calG$.

Set $H:=\Psi_\ell(\Gal_\calH)$.  We have $p_2(H)=(\OO/\ell\OO)^\times/\OO^\times$ since $\psi_{k,\ell}$ is surjective.  We have $p_1(H) = \rho_{E,\ell}(\Gal_\calH)$.   Since $\rho_{E,\ell}$ is surjective, $\calH/\QQ$ is solvable, and $\SL_2(\FF_\ell)$ is perfect, we have $\rho_{E,\ell}(\Gal_\calH) \supseteq \SL_2(\FF_\ell)$.    From our claim, we deduce that $\Psi_\ell(\Gal_\calH)=H=\calG$.
\end{proof}

\section{Proof of Theorem~\ref{T:main a}}  \label{S:main a}
Assume that GRH holds.  Fix a non-CM elliptic curve $E/\QQ$ and an integer $a$.   

For each prime $\ell$, we have constructed a representation $\rho_{E,\ell}\colon \Gal_\QQ \to \GL_2(\FF_\ell)$.    Let $I$ be the set of primes in the interval $[y, 2y]$, where $y$ is a fixed real number that satisfies $c\leq y \leq x$ for some constant $c$ depending on $E$.  We will make a more specific choice of $y$ later on.  After increasing $c$, we may assume that $I$ is non-empty and that $\rho_{E,\ell}$ is surjective for all primes $\ell\in I$.

For each prime $\ell$, define 
\[
P_{E,a}(x,\ell):= \#\{p\leq x : p\nmid N_E,\, a_p(E)=a \text{ and $\ell$ splits in $\QQ(\pi_p)$}\}.
\] 
Using our GRH assumption, Lemma~4.4 of \cite{MR935007} shows that
\begin{equation} \label{E:PEa break}
P_{E,a}(x) \ll_E \max_{\ell \in I}\, P_{E,a}(x,\ell).\\
\end{equation}

Now take any prime $\ell \in I$.  Set $L:=\QQ(E[\ell])$; it is the fixed field in $\Qbar$ of $\ker \rho_{E,\ell}$.   Using $\rho_{E,\ell}$, we may identify the Galois group $\Gal(L/\QQ)$ with $G:=\GL_2(\FF_\ell)$.   Define
\[
C:= \{A \in G : \tr(A)\equiv a \bmod{\ell} \text{ and } \tr(A)^2- 4\det(A) \in \FF_\ell \text{ is a square}\};
\]
it is a subset of $G$ that is stable under conjugation.

\begin{lemma}
We have $P_{E,k}(x,\ell) \leq \pi_C(x,L/\QQ) +1$.
\end{lemma}
\begin{proof}
Take any prime $p\nmid N_E \ell$ such that $a_p(E)=a$ and $\ell$ splits in $\QQ(\pi_p)$.  The representation $\rho_{E,\ell}$ is unramified at $p$ and we have $\tr(\rho_{E,\ell}(\Frob_p))\equiv a_p(E) = a$ and $\det(\rho_{E,\ell}(\Frob_p))\equiv p$ modulo $\ell$. 

Since $\ell$ splits in $\QQ(\pi_p)\cong \QQ((a_p(E)^2-4p)^{1/2}),$ we find that the image of $a_p(E)^2-4p$ in $\FF_\ell$ is a square.   Therefore, $\tr(\rho_{E,\ell}(\Frob_p))^2-4\det(\rho_{E,\ell}(\Frob_p)) \in \FF_\ell$ is a square.

We have thus shown that $\rho_{E,\ell}(\Frob_p) \subseteq C$.  The bound $P_{E,k}(x,\ell) \leq \pi_C(x,L/\QQ) +1$ is now clear; we have added 1 to take into account the excluded prime $p=\ell$.
\end{proof}

Let $B$ be the group of upper triangular matrices in $G$. 

\begin{lemma} \label{L:second a}
We have $P_{E,a}(x,\ell) \leq\widetilde\pi_{C\cap B}(x,L/L^B) +1$.
\end{lemma}
\begin{proof}
Observe that every conjugacy class of $G$ in $C$ contains an element from $B$.   Lemma~\ref{L:invariance C}(\ref{L:invariance C i})  implies that $P_{E,a}(x,\ell) \leq \widetilde\pi_{C}(x,L/\QQ)+1 \leq \widetilde\pi_{C\cap B}(x,L/L^B) +1$.
\end{proof}

 Let $U$ be the subgroup of $B$ consisting of the upper triangular matrices whose diagonal entries are both $1$.  The group $U$ is normal in $B$ and $B/U$ is abelian.  Let $C'$ be the image of $C\cap B$ in $B/U = \Gal(L^U/L^B)$.
 
\begin{lemma} \label{L:third a}
We have $P_{E,a}(x,\ell) \leq\widetilde\pi_{C'}(x,L^U/L^B) +1$.
\end{lemma}
\begin{proof}
We have $U\cdot (C\cap B) = C \cap B$.  Lemma~\ref{L:invariance C}(\ref{L:invariance C ii}) implies that $\widetilde\pi_{C\cap B}(x,L/L^B) = \widetilde\pi_{C'}(x,L^U/L^B)$.  Therefore, $P_{E,a}(x,\ell) \leq\widetilde\pi_{C'}(x,L^U/L^B) +1$ by Lemma~\ref{L:second a}.
\end{proof}

Before applying our Chebotarev bound to $\widetilde\pi_{C'}(x,L^U/L^B)$, we first bound some of the terms that will occur.

\begin{lemma} \label{L:a comps}
We have $|C'|\ll \ell$, $|C'|/|B/U| \ll 1/\ell$, $[L^B:\QQ] \ll \ell$ and $\log M(L^U/L^B) \ll_E \log \ell$.
\end{lemma}
\begin{proof}
We have $|G|=(\ell-1)^2(\ell+1)\ell \asymp \ell^4$, $|B|=(\ell-1)^2\ell \asymp \ell^3$, $|U|=\ell$ and $|B/U|=(\ell-1)\ell \asymp \ell^2$.  We have $[L^B:\QQ] = [G:B] \asymp \ell$.  The map $(\FF_\ell^\times)^2 \to B,$ $(b_1,b_2)\mapsto \big(\begin{smallmatrix} b_1 & 0\\ 0 & b_2\end{smallmatrix} \big)$ induces an isomorphism $(\FF_\ell^\times)^2 \xrightarrow{\sim} B/U$.  Therefore, 
\[
|C'|=|\{(b_1,b_2) \in \FF_\ell^2: b_1 b_2 \neq 0,\, b_1+b_2\equiv a \bmod{\ell}\}| \leq \ell.
\]
We thus have $|C'|/|B/U| \ll \ell/\ell^2=1/\ell$.  

If a prime $p$ ramifies in $L$, then $p|N_E\ell$.  By Proposition~4' of \cite{MR644559}, we have $\log(d_{L^B}^{1/[L^B:\QQ]})\ll_E \log(\ell\cdot [L^B:\QQ]) \ll \log \ell$.  Using the above bounds together, we find that $M(L^U/L^B)\ll_E \ell^2\cdot \log \ell \cdot \ell$ and hence $\log M(L^U/L^B) \ll_E \log \ell$.
\end{proof}

The AHC conjecture holds for the extension $L^U/L^B$ since its Galois group $B/U$ is abelian.  By Theorem~\ref{T:Chebotarev upper bound} and our GRH assumption, we have
\begin{align*}
\widetilde\pi_{C'}(x,L^U/L^B) & \ll \frac{|C'|}{|B/U|} \frac{x}{\log x} + |C'|^{1/2} [L^B:\QQ] \frac{x^{1/2}}{\log x} \log M(L^U/L^B)\\
& \ll_E \frac{1}{\ell} \frac{x}{\log x} + \ell^{1/2}\cdot \ell \cdot \frac{x^{1/2}}{\log x} \log \ell,
\end{align*}
where the last line uses Lemma~\ref{L:a comps}.  Lemma~\ref{L:third a} and $\ell \in [y,2y]$ implies that
\[
P_{E,a}(x,\ell) \ll_E  \frac{1}{y} \frac{x}{\log x} + y^{3/2}  \frac{x^{1/2}}{\log x} \log y.
\]
Since this holds for all $\ell\in I$, the inequality (\ref{E:PEa break}) gives
\[
P_{E,a}(x) \ll_E  \frac{1}{y} \frac{x}{\log x} + y^{3/2}  \frac{x^{1/2}}{\log x} \log y.
\]
Take $y := c' \cdot x^{1/5}/(\log x)^{2/5}$, where $c'$ is a constant chosen large enough to ensure that $y\geq c$ for all $x\geq 2$.   With this choice of $y$, we obtain the bound $P_{E,a}(x) \ll_E x^{4/5}/(\log x)^{3/5}$.\\

Finally consider the case where $a=0$.  Take any $\ell \in I$ and keep notation as above.  Let $H$ be the subgroup of $B$ consisting of the matrices whose eigenvalues are both equal; it is a normal subgroup of $B$ and we have $H\cdot (C \cap B) = C \cap B$ (multiplying a trace $0$ matrix by a scalar does not change the trace).   Lemma~\ref{L:invariance C}(\ref{L:invariance C ii}) implies that $\widetilde\pi_{C\cap B}(x,L/L^B) = \widetilde\pi_{C''}(x,L^H/L^B)$, where $C''$ is the image of $C\cap H$ in $B/H$.  Therefore, $P_{E,a}(x,\ell) \leq\widetilde\pi_{C''}(x,L^H/L^B) +1$ by Lemma~\ref{L:second a}.    Arguing as above, and using $|B/H|=\ell-1$ and $|C''|=1$, we have
\[
P_{E,0}(x)\ll \max_{\ell \in I}\, P_{E,0}(x,\ell) \ll \max_{\ell\in I}  \Big( \frac{1}{\ell} \frac{x}{\log x} + 1^{1/2}\cdot \ell \cdot \frac{x^{1/2}}{\log x} \log \ell \Big).
\]
Choosing $y \asymp x^{1/4}/(\log x)^{1/2}$, we deduce that $P_{E,0}(x) \ll  x^{3/4}/(\log x)^{1/2}$.

\section{Proof of Theorem~\ref{T:main k}}  \label{S:main k}

Fix a non-CM elliptic curve $E$ over $\QQ$ and an imaginary quadratic field $k$.  Let $\OO$ be the ring of integers of $k$. Let $\calH$ be the Hilbert class field of $k$ and let $h_k$ be the class number of $k$.    Fix a prime $\ell \geq 5$ such that $\rho_{E,\ell}$ is surjective and $\ell$ \emph{splits} in $k$; we will make a more specific choice later.\\

In \S\ref{SS:mixed}, we constructed a Galois representation
\[
\Psi_\ell\colon \Gal_{\calH} \to \calG,
\]
where $\calG:= \{ (A,u)\in\GL_2(\FF_\ell) \times \left((\OO/\ell\OO)^\times/\OO^\times \right) : \det(A)= N_{k/\QQ}(u) \}$.  The representation $\Psi_\ell$ is surjective by Lemma~\ref{L:intersection}.     Let $L$ be the fixed field in $\bbar{\calH}$ of $\ker \Psi_\ell$.   Using $\Psi_\ell$, we will identify the Galois group $\Gal(L/\calH)$ with $\calG$.  

Recall that the trace map $\Tr_{k/\QQ}\colon k\to \QQ$ induces a linear map $\Tr_{k/\QQ}\colon \OO/\ell\OO \to \FF_\ell$.   Define the set
\[
\calC := \{ (A,u) \in \calG :  \tr(A) \in \Tr_{k/\QQ}(u),\,  \tr(A)^2-4\det(A)\in \FF_\ell \text{ is a square}\};
\]
it is a subset of $\calG$ stable under conjugacy.   We now give a useful bound for $P_{E,k}(x)$.

\begin{lemma} \label{L:first bound k}
We have $P_{E,k}(x) \leq \frac{1}{h_k} \,\pi_\calC(x,L/\calH) + 4$.
\end{lemma}
\begin{proof}
Take any prime $p\nmid  N_E\ell$ for which $E$ has ordinary reduction at $p$ and for which $k\cong \QQ(\pi_p)$.  By Lemma~\ref{L:LT2 key}, the prime $p$ splits completely in $\calH$.   Let $\P \in \Sigma_\calH$ be any of the $h_k$ primes that divide $\pi_\p \OO$.  By Lemma~\ref{L:LT2 key}, we have $\tr(\rho_{E,\ell}(\Frob_\P))\in \Tr_{k/\QQ}(\psi_{k,\ell}(\Frob_\P))$.

Since $k\cong \QQ(\pi_p)$ and $\ell$ splits in $k$, the polynomial $x^2-a_p(E)x +p$ will factor modulo $\ell$.  Therefore, the image of $a_p(E)^2- 4p$ in $\FF_\ell$ is a square.  Since $p$ splits completely in $\calH$, we have $\tr(\rho_{E,\ell}(\Frob_\P))=\tr(\rho_{E,\ell}(\Frob_p))$ and $\det(\rho_{E,\ell}(\Frob_\P))=\det(\rho_{E,\ell}(\Frob_p))$.  Therefore, $\tr(\rho_{E,\ell}(\Frob_\P))^2- 4\det(\rho_{E,\ell}(\Frob_\P)) \equiv a_p(E)^2-4p \pmod{\ell}$ is a square.

We have verified that $\Psi_\ell(\Frob_\P) \subseteq \calC$ for each of the $h_k$ primes $\P$ dividing $\pi_p\OO$.   So the set
\[
\{p\leq x: p\nmid N_E,\, \QQ(\pi_p)\cong k \} - (S\cup\{\ell\})
\]
has cardinality at most $\frac{1}{h_k} \pi_C(x,L/\calH)$, where $S$ is the set of primes $p\nmid N_E$ for which $E$ has supersingular reduction at $p$ and $\QQ(\pi_p)\cong k$.  It thus suffices to show that $|S|\leq 3$.   

Take any prime $p\in S$ with $p\geq 5$.   Since $E$ has supersingular reduction at $p\geq 5$, we have $a_p(E)=0$.   Therefore, $k$ is isomorphic to $\QQ(\pi_p)\cong \QQ(\sqrt{-p})$.    So if $p\in S$, then $p$ is $2$, $3$ or the unique prime (if it exists) such that $k\cong \QQ(\sqrt{-p})$.  Therefore, $|S|\leq 3$.
\end{proof}

Let $B$ be the group of upper triangular matrices in $\GL_2(\FF_\ell)$.   Define
\[
\Bb := \{ (A,u) \in \calG : A \in B \};
\]
it is a subgroup of $\calG$.  We can identify $\Bb$ with the Galois group $\Gal(L/L^\Bb)$.

\begin{lemma} \label{L:second bound k}
We have $P_{E,k}(x) \leq \frac{1}{h_k} \,\widetilde\pi_{\calC\cap \Bb}(x,L/L^\Bb) + 4$.
\end{lemma}
\begin{proof}
 Any matrix $A\in \GL_2(\FF_\ell)$ with $\tr(A)^2-4\det(A)\in \FF_\ell$ a square is conjugate to a matrix in $B$.  Therefore, every element of $\calC$ is conjugate in $\calG$ to some element of $\Bb$.  By Lemma~\ref{L:invariance C}(\ref{L:invariance C i}), we have $\widetilde{\pi}_\calC(x,L/\calH) \leq \widetilde{\pi}_{C\cap \Bb}(x,L/L^\calB)$.  The lemma now follows from Lemma~\ref{L:first bound k} and the easy bound $\pi_\calC(x,L/\calH)\leq \widetilde\pi_\calC(x,L/\calH)$.
\end{proof}

Let $\calU$ be the image of the group 
\[
\{(A,a)\in \GL_2(\FF_\ell) \times \FF_\ell^\times: \text{the eigenvalues of $A$ are both $a$}\}
\]
in $\calG$ (we can identify $\FF_\ell^\times$ with a subgroup of $(\OO/\ell\OO)^\times$ since $\FF_\ell$ is a subalgebra of $\OO/\ell\OO$).  The group $\calU$ is normal in $\Bb$ and $\Bb/\calU$ is an abelian group.   We can identify $\Bb/\calU$ with the Galois group $\Gal(L^\calU/L^\Bb)$.  Let $\calC'$ be the image of $\calC\cap \Bb$ under the homomorphism $\Bb\to \Bb/\calU$; it is stable under conjugacy in $\Bb/\calU$.

\begin{lemma} \label{L:third bound k}
We have $P_{E,k}(x) \leq \frac{1}{h_k} \,\widetilde\pi_{\calC'}(x,L^\calU/L^\Bb) + 4$.
\end{lemma}
\begin{proof}
Observe that $\calU\cdot (\calC\cap \Bb) = \calC \cap \Bb$; whether an element $(A,u) \in \Bb$ belongs to $\calC$ depends only on $u$ and the eigenvalues of $A$, and that $\tr(A)\in \Tr_{k/\QQ}(u)$ remains true if $A$ and $u$ are multiplied by a common scalar in $\FF_\ell^\times$.  Lemma~\ref{L:invariance C}(\ref{L:invariance C ii}) implies that $\widetilde{\pi}_{\calC\cap \Bb}(x,L/L^\Bb) = \widetilde{\pi}_{C'}(x,L^\calU/L^\Bb)$.   The lemma then follows from Lemma~\ref{L:second bound k}.   
\end{proof}

Since $L^\calU/L^\calB$ is an abelian extension, we can now apply our Chebotarev bounds to obtain bounds for $P_{E,k}(x)$.   We first bound some terms that will show up.

\begin{lemma} \label{L:specific discriminant bounds}
\begin{romanenum}
\item
We have $|\calG| \asymp \ell^5$, $|\Bb| \asymp \ell^4$ and $|\calU|\asymp \ell^2$.
\item 
We have $|\calC'|\ll \ell$ and $|\calC'|/|\Bb/\calU| \ll  1/\ell$.  
\item
We have $[L^\Bb: \QQ] \ll h_k \ell$.
\item
We have $\log M(L^\calU/L^\Bb) \ll_E \log (d_k\ell)$.
\end{romanenum}
\end{lemma}
\begin{proof}
Since $\ell$ splits in $k$, we have $(\OO/\ell\OO)^\times \cong \FF_\ell^\times \times \FF_\ell^\times$.  Therefore,
\[
|\calC\cap \Bb| \leq |\{ (A,b,c) \in B \times \FF_\ell^\times \times \FF_\ell^\times : \det(A)=bc \text{ and } \tr(A)=b+c \}| \leq 2 |B|,
\]
where the last inequality uses that $x^2-\tr(A)x +\det(A)$ has at most two roots $b,c\in \FF_\ell$.   We thus have $|\calC \cap \Bb| \leq 2 \ell^3$.  Since $\calU\cdot (\calC\cap \Bb) = \calC \cap \Bb$, we have $|\calC'| = |\calC \cap \Bb|/|\calU| \leq 2\ell^3/|\calU| \ll \ell$.   We have $|\calC'|/|\Bb/\calU| \ll  1/\ell$ since $|\Bb/\calU| \asymp \ell^2$.  We have $[L^\Bb:\QQ]=[\calH:\QQ][L^\Bb:\calH]= 2h_k\cdot [\calG:\Bb]$, so $[L^\Bb:\QQ] \ll h_k \ell$.

Let $\calP$ be the set of rational primes $p$ divisible by some $\P\in \Sigma_\calH$ that ramifies in $L$.    Each prime in $\calP$ divides $N_E \ell$.    We have $[\Bb:\calU]\ll \ell^2$, so
\[
\log M(L^\calU/L^\Bb) \leq \log\big([\Bb:\calU] d_{L^\Bb}^{1/[L^\Bb:\QQ]}\cdot N_E \ell) \ll_E [L^\Bb:\QQ]^{-1} \log(d_{L^\Bb}) + \log \ell.
\]
It thus suffices to prove that $[L^\Bb:\QQ]^{-1} \log(d_{L^\Bb}) \ll_E \log(d_k\ell)$.  By Proposition~4 of \cite{MR644559}, we have
\[
[L^\Bb:\QQ]^{-1} \log(d_{L^\Bb}) \leq [\calH:\QQ]^{-1} \log(d_\calH) +  {\sum}_{p\in \calP} \log p + |\calP| \log([L^\Bb:\calH]).
\]
We have $[L^\Bb:\calH]\ll \ell$, so $[L^\Bb:\QQ]^{-1} \log(d_{L^\Bb}) \ll_E [\calH:\QQ]^{-1} \log(d_\calH) + \log \ell$.  Since $\calH/k$ is unramified,  Proposition~4 of \cite{MR644559} implies that $[\calH:\QQ]^{-1} \log(d_\calH) = 2^{-1} \log d_k$ and hence $[L^\Bb:\QQ]^{-1} \log(d_{L^\Bb}) \ll_E \log(d_k\ell)$.
\end{proof}

\begin{lemma} \label{L:easy 0}
If $d_k > 4x$, then $P_{E,k}(x)=0$.
\end{lemma}
\begin{proof}
Suppose that $d_k > 4x$ and $P_{E,k}(x)>0$.  There is thus a prime $p\nmid N_E$ satisfying $p\leq x$ and $k\cong \QQ(\pi_p)$.  We have $\QQ(\sqrt{a_p(E)^2-4p})\cong \QQ(\sqrt{-d_k})$, and hence $d_k$ divides $a_p(E)^2-4p$ (the divisibility with respect to the prime $2$ uses that $a_p(E)^2-4p$ is congruent to $0$ or $1$ modulo $4$).  Therefore, $d_k \leq 4p-a_p(E)^2 \leq 4x$ which contradicts our assumption.
\end{proof}

By Lemma~\ref{L:easy 0}, we may assume that $d_k\leq 4x$; the desired bounds are trivial otherwise.

\subsection{Conditional bounds}

Assume that GRH holds.  

By Lemma~\ref{L:third bound k}, Theorem~\ref{T:Chebotarev upper bound} and Lemma~\ref{L:good approximation}, we have
\begin{align*}
P_{E,k}(x) &\leq \frac{1}{h_k} \widetilde{\pi}_{\calC'}(x,L^\calU/L^\Bb) + 4\\
& \ll \frac{1}{h_k} \Big( \frac{|\calC'|}{|\Bb/\calU|} \frac{x}{\log x} + |\calC'|^{1/2}\, [L^\Bb:\QQ] \frac{x^{1/2}}{\log x} \, \log M(L^\calU/L^\Bb) \Big)+4\\
& \ll_E   \frac{1}{h_k} \frac{1}{\ell} \frac{x}{\log x} + \ell^{3/2}\, \frac{x^{1/2}}{\log x} \log(d_k \ell).
\end{align*}
Using Lemma~\ref{L:specific discriminant bounds} and $d_k\leq 4x$, we find that
\[
P_{E,k}(x) \ll_E \frac{1}{h_k} \frac{1}{\ell} \frac{x}{\log x} + \ell^{3/2}\, \frac{x^{1/2}}{\log x} \log(x\ell).
\]
We still need to choose our prime $\ell$.

\begin{lemma} \label{L:GRH for split}
Assuming GRH, there is an absolute constant $\gamma>0$ such that if $y \geq \gamma (\log d_k)^2$, then there exists a prime $\ell$ in the interval $[y,2y]$ that splits completely in $k$.
\end{lemma}
\begin{proof} 
This follows from Corollary~\ref{C:cheb existence} with the extension $k/\QQ$.
\end{proof}

Define 
\[
y:= \begin{cases}
C \cdot h_k^{-2/5} \cdot {x^{1/5}}/{(\log x)^{2/5}} & \text{if $h_k \leq x^{1/2}/(\log x)^6$},\\
C \cdot (\log x)^2 & \text{otherwise},
\end{cases}
\]
where $C>0$ is some constant depending only on $E$.  In both cases, we have $y \geq C (\log x)^2$.  

Since $d_k \leq 4x$, we have, after possibly increasing $C$, that $y\geq \gamma(\log d_k)^2$ with $\gamma$ as in Lemma~\ref{L:GRH for split}.  By Lemma~\ref{L:GRH for split}, there is a prime $\ell \in [y,2y]$ that splits completely in $k$.  After possibly increasing the constant $C$ first, we may assume by Theorem~\ref{T:Serre} that $\rho_{E,\ell}$ is surjective and that $\ell \geq 5$.   With this prime $\ell$, we obtain the bound
\[
P_{E,k}(x) \ll_E \frac{1}{h_k} \frac{1}{\ell} \frac{x}{\log x} + \ell^{3/2}\, \frac{x^{1/2}}{\log x} \log(x\ell) \ll \frac{1}{h_k} \frac{1}{y} \frac{x}{\log x} + y^{3/2}\, \frac{x^{1/2}}{\log x} \log(xy).
\]
Since $y\ll_E x$, we have 
\begin{align} \label{E:balance}
P_{E,k}(x) \ll_E \frac{1}{h_k} \frac{1}{y} \frac{x}{\log x} + y^{3/2}\, x^{1/2}.
\end{align}
If $h_k \leq x^{1/2}/(\log x)^6$, and hence $y=Ch_k^{-2/5}  {x^{1/5}}/{(\log x)^{2/5}}$,  substituting for $y$ gives the bound 
\[
P_{E,k}(x) \ll_E h_k^{-3/5} x^{4/5}/(\log x)^{3/5};
\]
note that our $y$ was chosen so that both terms in (\ref{E:balance}) have the same magnitude.   In the case $h_k > x^{1/2}/(\log x)^6$, we obtain
\[
P_{E,k}(x) \ll_E \frac{1}{h_k}\frac{x}{(\log x)^3} + x^{1/2}(\log x)^{3} \leq 2 x^{1/2} (\log x)^3.
\]
The bound of Theorem~\ref{T:main k}(\ref{T:main k i}) follows by adding our two possible bounds for $P_{E,k}(x)$.

\subsection{Unconditional bounds}

Define 
\[
\displaystyle y:= C \frac{1}{h_k} \frac{\log x}{(\log \log x)^2},
\] 
where $C>0$ is a constant depending only on $E$ to be chosen later.  Suppose that there is a prime $\ell$ in the interval $[y,2y]$ such that $\ell$ splits in $k$, $\ell\geq 5$, and $\rho_{E,\ell}$ is surjective.    

The group $\calB/\calU = \Gal(L^\calU/L^\calB)$ is abelian.   By Theorem~\ref{T:effective}(\ref{T:effective ii}), there are absolute constants $b, c>0$ such that if $\log x \geq b [L^\calB:\QQ]  \log^2 M(L^\calU/L^\calB)$, then
\[
\pi_{\calC'}(x,L^\calU/L^\calB) \ll  \frac{|\calC'|}{|\calB/\calU|}\,  \frac{x}{\log x} + |\calC'|^{1/2}\, [L^\calB:\QQ]x \exp\bigg( - 
\frac{c(\log x)^{1/2}}{[L^\calB:\QQ]^{1/2}} \bigg) \log^2 (M(L^\calU/L^\calB)  x),
\]
where we have used that $\beta_{L^\calU}\leq 1$ if it exists.

By Lemma~\ref{L:specific discriminant bounds}, we have 
\[
[L^\calB:\QQ]  \log^2 M(L^\calU/L^\calB) \ll h_k \ell \log^2(d_k \ell) \ll h_k \ell \log^2(h_k \ell),
\]
where the last inequality use the Brauer-Seigel theorem.   Using Lemma~\ref{L:specific discriminant bounds}, we deduce that there are positive absolute constants $b'$ and $c'$ such that if $\log x \geq b' \cdot h_k \ell\cdot  \log^2(h_k \ell)$, then
\[
\pi_{\calC'}(x,L^\calU/L^\Bb) \ll_E  \frac{1}{\ell}\,  \frac{x}{\log x} + h_k \ell^{3/2} x \exp\bigg( - c' \sqrt{\frac{\log x}{h_k\ell}}\bigg)\log^2 (h_k\ell \cdot x).
\]
Using that $\ell \in [y,2y]$, we have
\[
\pi_{\calC'}(x,L^\calU/L^\Bb) \ll_E  h_k  \frac{x (\log \log x)^2}{(\log x)^2} + \frac{1}{\sqrt{h_k}}\cdot  \frac{(\log x)^{3/2}}{(\log\log x)^3}\cdot x \exp\bigg( - \frac{c'}{2\sqrt{C}} \log \log x \bigg)(\log x)^2.
\]
By taking our constant $C>0$ sufficiently small, we find that $\pi_{\calC'}(x,L^\calU/L^\Bb) \ll_E h_k \cdot \frac{x (\log \log x)^2}{(\log x)^2}$.  By Lemmas~\ref{L:good approximation} and \ref{L:specific discriminant bounds}, we find that  $\tilde\pi_{\calC'}(x,L^\calU/L^\Bb) \ll_E h_k \cdot {x (\log \log x)^2}/{(\log x)^2}$.    Therefore,
\[
P_{E,k}(x) \ll_E \frac{x (\log \log x)^2}{(\log x)^2}
\]
by Lemma~\ref{L:third bound k}.

Finally, we now need to know that such a prime $\ell$ exists; at least if $x$ is sufficiently large.  By the Chebotarev density theorem and Theorem~\ref{T:Serre}, there is a constant $\gamma \geq 1$, depending on $E$ and $k$, such that if $y\geq\gamma$, then there is a prime $\ell \in [y,2y]$ for which $\ell$ splits in $k$ and $\rho_{E,\ell}$ is surjective.  So for $x$ sufficiently large, we will have $y\geq \gamma$ and the desired prime $\ell$ exists.  (One could make this explicit by using an effective version of the Chebotarev density theorem.)  This completes the proof of Theorem~\ref{T:main k}(\ref{T:main k ii}).

\section{Proof of Corollary~\ref{C:D theorem}} \label{S:D}

Let $\calD_E(x)$ be the set of imaginary quadratic extensions $k$ of $\QQ$ for which there exists a prime $p\leq x$ with $\QQ(\pi_p) \cong k$; note that $|\calD_E(x)|=D_E(x)$.  We start with the identity
\[
\pi(x) = |\{p \leq x: p| N_E\}| + {\sum}_{k\in \calD_E(x)} P_{E,k}(x),
\]
Theorem \ref{T:main k}(\ref{T:main k i}) then implies that
\begin{align*} 
x/\log x \ll_E \sum_{k\in \calD_E(x)} P_{E,k}(x)  \ll_E &  \sum_{k\in \calD_E(x)}\Bigl( \frac{1}{h_k^{3/5}}\frac{x^{4/5}}{(\log x)^{3/5}} + x^{1/2}(\log x)^3 \Bigr)\\
 =& \sum_{k\in \calD_E(x)} \frac{1}{h_k^{3/5}} \,\cdot \, \frac{x^{4/5}}{(\log x)^{3/5}} + D_E(x) x^{1/2}(\log x)^3.
\end{align*}
Using GRH, one can show that, $h_k \gg d_k^{1/2}/\log{d_k}.$  By Lemma~\ref{L:easy 0}, we have $d_k\leq 4x$ for all $k\in \calD_E(x)$.  Using these bounds, we have:
\begin{align*} 
\sum_{k\in \calD_E(x)} \frac{1}{h_k^{3/5}} & \ll  \sum_{k\in\calD_E(x)} \frac{(\log d_k)^{3/5}}{d_k^{3/10}} \ll  \sum_{k\in \calD_E(x)} \frac{1}{d_k^{3/10}} (\log x)^{3/5}
\leq  \sum_{d=1}^{ D_E(x) } \frac{1}{d^{3/10}} (\log x)^{3/5}.
\end{align*}
Therefore, $\sum_{k\in \calD_E(x)} {h_k^{-3/5}}  \ll  D_E(x)^{7/10} (\log x)^{3/5}$.
Combining with our previous inequality, we have
\begin{align*} 
x/\log x  &\ll_E   D_E(x)^{7/10} x^{4/5} \, + \, D_E(x)x^{1/2}(\log x)^3.
\end{align*} 
Therefore, we have $x/\log x \ll_E D_E(x)^{7/10} x^{4/5}$ or $x/\log x \ll_E D_E(x)x^{1/2}(\log x)^3$.   Equivalently, we have $D_E(x) \gg_E x^{2/7}/(\log x)^{10/7}$ or $D_E(x) \gg_E x^{1/2}/(\log x)^4$.  We conclude that $D_E(x)\gg_E {x^{2/7}}/{(\log x)^{10/7}}$ since this is the weaker of the two possible bounds.

\begin{remark}
If we had instead used the bound $P_{E,k}(x) \ll_E {x^{4/5}}/{(\log x)^{3/5}}$, then we would have deduced that $D_E(x)\gg_E {x^{1/5}}/{(\log x)^{2/5}}$.  Thus the factor $h_k^{-3/5}$ occuring in our bound of $P_{E,k}(x)$ gives a significant improvement.
\end{remark}

\section{Proof of Theorem~\ref{T:smoothed}} \label{S:smoothed proof}

Fix a real number $x\geq 2$. For each class function $\varphi\colon G\to \CC$, we define
\[
\Theta_\varphi(x) := \sum_{\p \in \Sigma_K, m\geq 1}  \varphi(\Frob_\p^m) \log N(\p)  \cdot f(N(\p)^m/x).
\]
We first estimate $\Theta_\chi(x)$ for irreducible characters $\chi\colon G \to \CC$.

\begin{lemma} \label{L:key bounds}
For any irreducible character $\chi$ of $G$, we have
\[
\Theta_\chi(x) = \delta_\chi \cdot x \int_0^\infty f(t)\, dt +  O_f\Big(\chi(1)[K:\QQ] x^{1/2} \log M(L/K)\Big),
\]
where $\delta_\chi=1$ if $\chi=1$ and $\delta_\chi=0$ otherwise.  
\end{lemma}
\begin{proof}
Let $L(s,\chi)$ be the Artin $L$-function arising from $\chi$; for background on Artin $L$-functions see \cite{MR0447187}.    By our AHC assumption, the series $L(s,\chi)$ extends to a function analytic everywhere except at $s=1$ when $\chi=1$.  We have $\ord_{s=1} L(s,\chi)=-\delta_\chi$.  

Define $A_\chi:=d_K^{\chi(1)}\cdot N(\calF_\chi)$, where $\calF_\chi \subseteq \OO_K$ is the Artin conductor corresponding to $\chi$.  We define the completed $L$-function $\Lambda(s,\chi) := A_\chi^{s/2} \gamma_\chi(s) L(s,\chi)$, where $\gamma_\chi(s)$ is a certain product of $\Gamma$-factors.   See \cite{MR0447187}*{p.12} for the precise definition of $\gamma_\chi$; we simply note that there are explicit positive integers $a$ and $b$ with $a+b = \chi(1)[K:\QQ]$ such that
\[
\gamma_\chi(s)=(\pi^{-s/2} \Gamma(\tfrac{s}{2}))^a \cdot (\pi^{-(s+1)/2} \Gamma(\tfrac{s+1}{2}))^b.
\] 
The \defi{functional equation} for $\Lambda(s,\chi)$ says that
\[
\Lambda(s,\chi) = W_\chi \cdot \Lambda(1-s,\bbar{\chi})
\]
for some $W_\chi \in \CC^\times$ with absolute value $1$.  The logarithmic derivative of the Artin $L$-series of $L(s,\chi)$ is
\[
\frac{L'}{L}(s,\chi) = -\sum_{\p \in \Sigma_K} \log N(\p) \sum_{m\geq 1} \chi(\Frob_\p^m) N(\p)^{-ms}. 
\]
In particular, $-\frac{L'}{L}(s,\chi) = \sum_{n\geq 1} \Lambda_\chi(n) n^{-s}$, where 
\[
\Lambda_\chi(n) := \log n  \sum_{\substack{\p\in\Sigma_K, m\geq1\\ N(\p)^m = n}}    \frac{1}{m} \chi(\Frob_\p^m). 
\]

Let $\varphi\colon (0,+\infty) \to \CC$ be a smooth function with compact support. The Mellin transform of $\varphi$ is
\[
\widehat{\varphi}(s) := \int_0^\infty \varphi(t) t^s\, \frac{dt}{t}.
\]
Define the function $\psi\colon (0,+\infty) \to \CC$ by $\psi(t):=t^{-1}\varphi(t^{-1})$.   The \defi{explicit formula}, as given by Iwaniec and Kowalski in Theorem~5.11 of \cite{MR2061214}, says that
\begin{align}  \label{E:explicit formula}
\sum_{n\geq 1} \Big( \Lambda_\chi(n) \varphi(n) + \bbar{\Lambda_\chi(n)} \psi(n) \Big) =&\,\, \varphi(1) \log A_\chi + \delta_\chi \int^\infty_0 \varphi(t) \, dt  \\
&+  \frac{1}{2\pi i} \int_{(1/2)} \Big(\frac{\gamma_\chi'}{\gamma_\chi}(s) + \frac{\gamma_{\bbar{\chi}}'}{\gamma_{\bbar{\chi}}}(1-s) \Big) \widehat{\varphi}(s)\, ds - \sum_\rho \widehat{\varphi}(\rho),  \notag
\end{align}
where the sum is over the zeros $\rho$ of $L(s,\chi)$, with multiplicity, for which $0\leq \operatorname{Re}(\rho) \leq 1$.   The explicit formula in \cite{MR2061214} is given for a general $L$-function that satisfies certain properties; they are all known to hold for Artin $L$-function except for the analytic continuation which holds by our ongoing AHC assumption.  \\

We now take $\varphi\colon (0,+\infty)\to \CC$ to be the function $\varphi(t)=f(t/x)$.  Observe that 
\[
\Theta_\chi(x) = \sum_{n\geq 1} \Lambda_\chi(n) \varphi(n)
\] 
and that $\delta_\chi \int^\infty_0 \varphi(t) \, dt = \delta_\chi x \int^\infty_0 f(t) \, dt$; it thus remains to bound the other terms occurring in (\ref{E:explicit formula}).

We first bound $\sum_{n\geq 1} \bbar{\Lambda_\chi(n)} \psi(n)$.   If $\Lambda_\chi(n)$ is non-zero, then $n$ is a prime power.   Since there are at most $[K:\QQ]$ primes $\p\in \Sigma_K$ dividing a fixed rational prime, we have $|\Lambda_\chi(n)| \leq \log n \cdot \chi(1) [K:\QQ]$.  There is a number $c>0$, depending only on $f$, such that $f(t)=0$ for $t\leq c$.  If $n \geq c^{-1}$, then $1/(xn) \leq c$ and hence $\psi(n)=n^{-1}f(1/(xn))=0$.  Therefore,
\[
\sum_{n\geq 1} \bbar{\Lambda_\chi(n)} \psi(n) \ll\sum_{n \leq c^{-1}} |\Lambda_\chi(n)| \, |\psi(n)|
\leq   \chi(1) [K:\QQ] \sum_{n\leq c^{-1}} \log n \cdot \sup_{t\in \RR}|f(t)| \ll_f \chi(1)[K:\QQ].
\]

We have $\log A_\chi \ll \chi(1) [K:\QQ] \log M(L/K)$ by Proposition~2.5 of \cite{MR935007}.  Therefore, $\varphi(1) \log A_\chi  \ll_f \chi(1) [K:\QQ] \log M(L/K)$.

For $y\in \RR$, we have
\[
\widehat{\varphi}(1/2+i y ) = \int_0^\infty f(t/x) t^{1/2+i y}\, \frac{dt}{t} = x^{1/2} \cdot x^{i y} \int^{\infty}_{-\infty} f(e^u) e^{u/2} e^{i\cdot u y} \, du,
\]
where we have made the substitution $t=xe^u$.   We have $\int^{\infty}_{-\infty} f(e^u) e^{u/2} e^{i\cdot u y} \, du \ll_f 1/(|y|+1)^2$ since $f(e^u)e^{u/2}$, and hence also its Fourier transform, is a Schwartz function.   Therefore, 
\[
\widehat{\varphi}(1/2+i y ) \ll_f  x^{1/2}/(|y|+1)^2.
\]
Using that $\frac{\Gamma}{\Gamma}'(1/2+i y) \ll \log(|y|+2)$ for all real $y$ (cf.~Lemma~6.1 of \cite{MR0447191}), we have 
\[
\tfrac{\gamma_\chi'}{\gamma_\chi}(1/2+iy) \ll \chi(1)[K:\QQ] x^{1/2}/(|y|+1)^2.
\]
Therefore,
\[
\int_{(1/2)} \Big(\tfrac{\gamma_\chi'}{\gamma_\chi}(s) + \tfrac{\gamma_{\bbar{\chi}}'}{\gamma_{\bbar{\chi}}}(1-s) \Big) \widehat{\varphi}(s)\, ds \ll_f   \chi(1)[K:\QQ] x^{1/2} \int^\infty_{-\infty} \tfrac{\log(|y|+2)}{(|y|+1)^2} \, dy \ll  \chi (1) [K:\QQ] x^{1/2}.
\]

We now bound the sum $\sum_\rho \widehat{\varphi}(\rho)$.  We have $\zeta_L(s) = \prod_{\chi} L(s,\chi)^{\chi(1)}$, where the product is over irreducible characters $\chi$ of $G$ and $\zeta_L$ is the Dedekind zeta function of $L$.    By assumption, AHC holds for $L/K$ and GRH holds for $L$, so we deduce that any zero $\rho$ of $L(s,\chi)$ with $0\leq \operatorname{Re}(\rho)\leq 1$ satisfies $\operatorname{Re}(\rho)=1/2$.   Therefore,
\[
\sum_\rho \widehat{\varphi}(\rho) \ll_f x^{1/2} \sum_{\rho=1/2+i y} 1/(|y|+1)^2.
\]
For each real number $t$, let $N(t,\chi)$ be the number of zeros $\rho=1/2+i y$ of $L(s,\chi)$, counted with multiplicity, such that $|t-y| \leq 1$.   From equation (3.5.5) and Proposition~2.5 of \cite{MR935007}, we have 
\[
N(t,\chi)  \ll \chi(1) [K:\QQ] \log M(L/K) + \chi(1)[K:\QQ] \log(|t|+2).
\]  
Therefore,
\begin{align*}
\sum_\rho \widehat{\varphi}(\rho) \ll_f x^{1/2} \sum_{n\in \ZZ} \frac{N(n,\chi)}{(|n|+1)^2} 
&\ll x^{1/2} \chi(1)[K:\QQ] \log M(L/K)  \sum_{n\in \ZZ} \frac{\log(|n|+2)}{(|n|+1)^2}\\
&\ll x^{1/2} \chi(1)[K:\QQ] \log M(L/K).
\end{align*}
Using the above bounds with (\ref{E:explicit formula}), we obtain the desired estimate for $\Theta_\chi(x) = \sum_{n\geq 1} \Lambda_\chi(n) \varphi(n)$.
\end{proof}

\begin{lemma} \label{L:alway Cheb}
Let $D$ be any subset of $G$ that is stable under conjugation.  Then
\[
\Big|\Theta_{\delta_D}(x) - \frac{|D|}{|G|} \Theta_1(x)\Big| \leq |D|^{1/2}\cdot \Big(\frac{1}{|G|} \sum_{\chi\neq 1} |\Theta_{\chi}(x)|^2\Big)^{1/2},
\]
where the sum is over the non-trivial irreducible characters $\chi$ of $G$.
\end{lemma}
\begin{proof}
We have $\Theta_{\delta_D}(x) - \frac{|D|}{|G|}  \Theta_1(x) = \sum_{C\subseteq D} (\Theta_{\delta_C}(x) - \frac{|C|}{|G|} \Theta_1(x))$, where the sum is over the conjugacy classes $C$ of $G$.  Using the triangle inequality and the Cauchy-Schwartz inequality, we find that $|\Theta_{\delta_D}(x) - \frac{|D|}{|G|}  \Theta_1(x)|$ is less than or equal to
\[
\sum_{C\subseteq D} \Big|\Theta_{\delta_C}(x) - \frac{|C|}{|G|} \Theta_1(x) \Big| \leq \Big(\sum_{C\subseteq D} |C|\Big)^{1/2} \Big( \sum_C \frac{1}{|C|} \Big| \Theta_{\delta_C}(x) - \frac{|C|}{|G|} \Theta_{1}(x) \Big|^2 \Big)^{1/2}.
\]
Since $\sum_{C\subseteq D} |C| = |D|$, it suffices to prove that
\begin{equation} \label{E:conj sum 1}
\sum_C \frac{1}{|C|} \Big| \Theta_{\delta_C}(x) - \frac{|C|}{|G|} \Theta_{1}(x) \Big|^2
= \frac{1}{|G|} \sum_{\chi\neq 1} |\Theta_{\chi}(x)|^2,
\end{equation}
where the first sum is over the conjugacy classes $C$ of $G$ and the second sum is over the non-trivial irreducible characters of $G$.

For $C$ and $\chi$ as above, let $\chi(C)$ be the common value of $\chi(g)$ with $g\in C$.    We have $\delta_C = \tfrac{|C|}{|G|} \sum_{\chi} \bbar{\chi(C)}\cdot \chi$, so by linearity $\Theta_{\delta_C}(x) = \frac{|C|}{|G|} \sum_\chi \bbar{\chi(C)} \,\Theta_{\chi}(x)$.  Therefore,
\begin{align*}
\sum_{C} \frac{1}{|C|} \Big| \Theta_{\delta_C}(x) - \frac{|C|}{|G|} \Theta_{1}(x) \Big|^2 
& = \sum_C \frac{1}{|C|} \Big| \frac{|C|}{|G|} \sum_{\chi\neq 1} \bbar{\chi(C)} \Theta_{\chi}(x) \Big|^2 \\
& = \frac{1}{|G|} \sum_C  \frac{|C|}{|G|}   \sum_{\chi\neq 1, \chi'\neq 1} \chi(C) \bbar{\chi'(C)} \Theta_{\chi}(x) \bbar{\Theta_{\chi'}(x)}
\end{align*}
Since $ \sum_C \frac{|C|}{|G|} \chi(C) \bbar{\chi'(C)} = \frac{1}{|G|} \sum_{g\in G} \chi(g) \chi'(g)$ is equal to $1$ if $\chi=\chi'$ and $0$ otherwise, we have
\[
\sum_{C} \frac{1}{|C|} \Big| \Theta_{\delta_C}(x) - \frac{|C|}{|G|} \Theta_{1}(x) \Big|^2 
= \frac{1}{|G|} \sum_{\chi\neq 1} \Theta_{\chi}(x) \bbar{\Theta_{\chi}(x)}
= \frac{1}{|G|} \sum_{\chi\neq 1} |\Theta_{\chi}(x)|^2. 
\]
This proves (\ref{E:conj sum 1}).
\end{proof}

\begin{lemma} \label{L:almost there}
Let $C$ be any subset of $G$ stable under conjugation.   Then
\[
\Theta_{\delta_C}(x) = \frac{|C|}{|G|}\,x \int^\infty_0 f(t) \, dt + O_f\Big( |C|^{1/2} [K:\QQ] x^{1/2} \log M(L/K) \Big).
\]
\end{lemma}
\begin{proof}
The lemma is trivial if $C=\emptyset$, so assume that $C\neq \emptyset$.  By Lemma~\ref{L:alway Cheb} and Lemma~\ref{L:key bounds}, we have
\begin{align*}
\Big|\Theta_{\delta_C}(x) - \frac{|C|}{|G|} \Theta_1(x)\Big| 
&\leq |C|^{1/2}\cdot \Big(\frac{1}{|G|} \sum_{\chi\neq 1} \big(\chi(1)[K:\QQ] x^{1/2} \log M(L/K)\big)^2\Big)^{1/2}\\
& \ll_f |C|^{1/2} [K:\QQ] x^{1/2} \log M(L/K) \cdot \big(\frac{1}{|G|}{\sum}_{\chi\neq 1} \chi(1)^2 \big)^{1/2}.
\end{align*}
Since $\sum_\chi \chi(1)^2 = |G|$, we have 
\[
\Theta_{\delta_C}(x) = \frac{|C|}{|G|} \Theta_1(x) +O_f\Big( |C|^{1/2} [K:\QQ] x^{1/2} \log M(L/K) \Big).
\]
The lemma follows by using Lemma~\ref{L:key bounds} with $\chi=1$ to estimate $\Theta_1(x)$.
\end{proof}

There is a constant $c>0$, depending only on $f$, such that $f(t)=0$ for all $t\geq c$.    In particular, we have $f(N(\p)^m/x)= 0$ if $N(\p)^m \geq  c x$.     Let $S(x)$ be the sum in the statement of Theorem~\ref{T:smoothed}.    One can readily check that
\[
0 \leq \Theta_{\delta_C}(x) - S(x) \leq (\tilde\pi_C(cx,L/K)-\pi_C(cx,L/K))\cdot  \log(cx) \cdot \max_{t\in \RR} |f(t)|.
\]
By Lemma~\ref{L:good approximation}, we have $S(x)=\Theta_{\delta_C}(x) + O_f( [K:\QQ] x^{1/2} \log M(L/K))$.  Theorem~\ref{T:smoothed} now follows directly from Lemma~\ref{L:almost there}.


\begin{bibdiv}
\begin{biblist}

\bib{MR2464027}{article}{
      author={Cojocaru, Alina~Carmen},
      author={David, Chantal},
       title={Frobenius fields for elliptic curves},
        date={2008},
        ISSN={0002-9327},
     journal={Amer. J. Math.},
      volume={130},
      number={6},
       pages={1535\ndash 1560},
      review={\MR{MR2464027 (2009k:11092)}},
}

\bib{MR2178556}{article}{
      author={Cojocaru, Alina~Carmen},
      author={Fouvry, Etienne},
      author={Murty, M.~Ram},
       title={The square sieve and the {L}ang-{T}rotter conjecture},
        date={2005},
        ISSN={0008-414X},
     journal={Canad. J. Math.},
      volume={57},
      number={6},
       pages={1155\ndash 1177},
  url={http://journals.cms.math.ca/ams/ams-redirect.php?Journal=CJM&Volume=57&FirstPage=1155},
      review={\MR{MR2178556 (2006e:11074)}},
}

\bib{MR1144318}{article}{
      author={Elkies, Noam~D.},
       title={Distribution of supersingular primes},
        date={1991},
        ISSN={0303-1179},
     journal={Ast\'erisque},
      number={198-200},
       pages={127\ndash 132 (1992)},
        note={Journ{\'e}es Arithm{\'e}tiques, 1989 (Luminy, 1989)},
      review={\MR{1144318 (93b:11070)}},
}

\bib{MR2061214}{book}{
      author={Iwaniec, Henryk},
      author={Kowalski, Emmanuel},
       title={Analytic number theory},
      series={American Mathematical Society Colloquium Publications},
   publisher={American Mathematical Society},
     address={Providence, RI},
        date={2004},
      volume={53},
        ISBN={0-8218-3633-1},
      review={\MR{MR2061214 (2005h:11005)}},
}

\bib{MR0447191}{incollection}{
      author={Lagarias, J.~C.},
      author={Odlyzko, A.~M.},
       title={Effective versions of the {C}hebotarev density theorem},
        date={1977},
   booktitle={Algebraic number fields: {$L$}-functions and {G}alois properties
  ({P}roc. {S}ympos., {U}niv. {D}urham, {D}urham, 1975)},
   publisher={Academic Press, London},
       pages={409\ndash 464},
      review={\MR{0447191 (56 \#5506)}},
}

\bib{MR0568299}{book}{
      author={Lang, Serge},
      author={Trotter, Hale},
       title={Frobenius distributions in {${\rm GL}_{2}$}-extensions},
      series={Lecture Notes in Mathematics, Vol. 504},
   publisher={Springer-Verlag},
     address={Berlin},
        date={1976},
        note={Distribution of Frobenius automorphisms in
  ${{\rm{G}}L}_{2}$-extensions of the rational numbers},
      review={\MR{MR0568299 (58 \#27900)}},
}

\bib{MR0447187}{incollection}{
      author={Martinet, J.},
       title={Character theory and {A}rtin {$L$}-functions},
        date={1977},
   booktitle={Algebraic number fields: {$L$}-functions and {G}alois properties
  ({P}roc. {S}ympos., {U}niv. {D}urham, {D}urham, 1975)},
   publisher={Academic Press, London},
       pages={1\ndash 87},
      review={\MR{0447187 (56 \#5502)}},
}

\bib{MR935007}{article}{
      author={Murty, M.~Ram},
      author={Murty, V.~Kumar},
      author={Saradha, N.},
       title={Modular forms and the {C}hebotarev density theorem},
        date={1988},
        ISSN={0002-9327},
     journal={Amer. J. Math.},
      volume={110},
      number={2},
       pages={253\ndash 281},
      review={\MR{935007 (89d:11036)}},
}

\bib{MR1694997}{incollection}{
      author={Murty, V.~Kumar},
       title={Modular forms and the {C}hebotarev density theorem. {II}},
        date={1997},
   booktitle={Analytic number theory ({K}yoto, 1996)},
      series={London Math. Soc. Lecture Note Ser.},
      volume={247},
   publisher={Cambridge Univ. Press, Cambridge},
       pages={287\ndash 308},
      review={\MR{1694997 (2000g:11094)}},
}

\bib{MR819838}{article}{
      author={Ribet, Kenneth~A.},
       title={On {$l$}-adic representations attached to modular forms. {II}},
        date={1985},
        ISSN={0017-0895},
     journal={Glasgow Math. J.},
      volume={27},
       pages={185\ndash 194},
      review={\MR{819838 (88a:11041)}},
}



\bib{Rouse-Thorner}{article}{
       author={Rouse, Jeremy},
       author={Thorner, Jesse},
       title={The explicit Sato-Tate conjecture and densities pertaining to Lehmer-type questions},
  	     journal={arXiv:1305.5283 [math.NT]},
        date={2013},
}


\bib{MR0387283}{article}{
      author={Serre, Jean-Pierre},
       title={Propri\'et\'es galoisiennes des points d'ordre fini des courbes
  elliptiques},
        date={1972},
        ISSN={0020-9910},
     journal={Invent. Math.},
      volume={15},
      number={4},
       pages={259\ndash 331},
      review={\MR{MR0387283 (52 \#8126)}},
}

\bib{MR0450380}{book}{
      author={Serre, Jean-Pierre},
       title={Linear representations of finite groups},
   publisher={Springer-Verlag},
     address={New York},
        date={1977},
        ISBN={0-387-90190-6},
        note={Translated from the second French edition by Leonard L. Scott,
  Graduate Texts in Mathematics, Vol. 42},
      review={\MR{MR0450380 (56 \#8675)}},
}

\bib{MR644559}{article}{
      author={Serre, Jean-Pierre},
       title={Quelques applications du th\'eor\`eme de densit\'e de
  {C}hebotarev},
        date={1981},
        ISSN={0073-8301},
     journal={Inst. Hautes \'Etudes Sci. Publ. Math.},
      number={54},
       pages={323\ndash 401},
      review={\MR{MR644559 (83k:12011)}},
}

\bib{MR1062334}{article}{
      author={Wan, Da~Qing},
       title={On the {L}ang-{T}rotter conjecture},
        date={1990},
        ISSN={0022-314X},
     journal={J. Number Theory},
      volume={35},
      number={3},
       pages={247\ndash 268},
      review={\MR{1062334 (91f:11079)}},
}

\end{biblist}
\end{bibdiv}


\end{document}